\documentclass{amsart}
\usepackage{xcolor}
\newtheorem{theorem}{Theorem}[section]
\newtheorem{lemma}[theorem]{Lemma}
\usepackage{ amssymb }
\usepackage{tikz}
\usepackage{float}
\newtheorem{definition}[theorem]{Definition}

\newtheorem{proposition}[theorem]{Proposition}
\newtheorem{addendum}[theorem]{Addendum}

\newtheorem{corollary}[theorem]{Corollary}
\theoremstyle{remark}
\newtheorem{claim}[theorem]{Claim}
\newtheorem{remark}[theorem]{Remark}
\usepackage{verbatim}
\numberwithin{equation}{section}
\DeclareMathOperator{\supp}{supp}
\DeclareMathOperator{\diff}{Diff}

\DeclareMathOperator{\topo}{top}
\DeclareMathOperator{\erg}{erg}
\DeclareMathOperator{\ph}{PH}
\DeclareMathOperator{\phc}{PHC}
\DeclareMathOperator{\mme}{MME}

\begin{document}

\title[Maximal entropy measures for topological skew products]{Maximal entropy measures for non-accessible
topological skew products}
%    author one information
\author{Ali Tahzibi }
\address{}
\curraddr{}
\email{tahzibi@icmc.usp.br}
\thanks{}
%   Remove any unused author tags.
\author{Richard cubas}
\address{}
\curraddr{}
\email{cubas.mat@usp.br}
\thanks{}

%    author two information



\date{}

\dedicatory{}

\begin{abstract}
In this paper we establish a dichotomy for the ergodic measures of maximal entropy for partially hyperbolic diffeomorphisms with one-dimensional compact center leaves which are virtually skew products over (transitive) Anosov homeomorphism. We prove that if the whole manifold is the unique minimal invariant set saturated by unstable foliation,  then either there exists a unique measure of maximal entropy which is non-hyperbolic or there are exactly two hyperbolic ergodic measures of maximal entropy.
\end{abstract}

\maketitle
\section{Introduction}
Consider a dynamical system $f:M \rightarrow M$, where $M$ is a compact metric space and $f$ is
continuous. The metric entropy $h_{\mu}(f)$ describes the complexity $f$ relative to an $f$-invariant probability measure $\mu$. In this setting, the topological entropy $h_{top}(f)$ is the supremum of $h_{\mu}(f)$ taken over all $f$-invariant Borel probability measures on $M$.
An invariant probability measure  whose entropy coincides with the topological entropy is called  maximal entropy measure (M.M.E). Such measures reflect the complexity level of the whole system, and  classical questions in ergodic theory concerning these measure are as follows: Is the set of maximal measures non empty? How many ergodic maximal measures does the system have, and where are they supported? What are their fine ergodic properties? \\

When $f : M \rightarrow  M$ is a diffeomorphism and $X \subset M$ is a topologically transitive locally maximal hyperbolic set, it is well-known that $f$ has a unique maximal entropy measure \cite{bowen2008equilibrium}. Recently, new approaches were developed for the study of maximal measures of non-uniformly hyperbolic maps and partially hyperbolic diffeomorphisms \cite{buzzi2018measures, sarig2013symbolic, buzzi2012maximal, viana2013physical}. For some recent progress in the more general setting of equilibrium states, see \cite{climenhaga2018equilibrium}.\\

In the present paper, we study  $C^{2}$ dynamically coherent partially hyperbolic diffeomorphisms on compact manifolds with $1$-dimensional center 
where center foliation is uniformly compact, i.e.  the Riemannian volume of the center leaves  is uniformly bounded from above. In this context,  F. Hertz, M. Hertz, A. Tahzibi, and R. Ures \cite{hertz2012maximizing} obtained the following dichotomy for accessible partially hyperbolic diffeomorphisms of 3-dimensional manifolds: either there is a unique entropy maximizing measure, this measure has the Bernoulli property and its center Lyapunov exponent is zero or, there are finitely many (more than one) ergodic entropy maximizing measures, all of them with nonzero center Lyapunov exponent.  When $M$ is a 3-dimensional nil-manifold,  Ures, Viana and J. Yang \cite{ures2021maximal} proved  the following dichotomy for measures of maximal entropy: either  there is a unique maximal measure,  or there exist exactly two ergodic maximal measures, both hyperbolic and whose center Lyapunov exponents have opposite signs.  A. Tahzibi and E. Rocha \cite{rocha2022number}  assuming the existence of a periodic leaf with Morse–Smale dynamics proved  a sharp upper bound for the number of maximal measures in terms of the number of sources and sinks of Morse–Smale dynamics.

In this paper we concentrate mainly on the non-accessible partially hyperbolic diffeomorphisms. We assume a minimality type assumption and prove a sharp dichotomy result: either unique non-hyperbolic measure of maximal entropy or exactly two hyperbolic maximal measures with opposite sign of center Lyapunov exponent. 

\begin{definition}
 A diffeomorphism is (strongly) partially hyperbolic if there is an invariant splitting of the tangent bundle: $T M=E^{s} \oplus E^{c} \oplus E^{u}$ such that $E^{s} \oplus\left(E^{c} \oplus E^{u}\right)$ and $\left(E^{s} \oplus E^{c}\right) \oplus E^{u}$ are dominated, $E^{s}$ is uniformly contracted, and $E^{u}$ is uniformly expanded.
\end{definition}

A partially hyperbolic diffeomorphism is dynamically coherent if both $E^c \oplus E^u$ and $E^c \oplus E^s$ are tangente to invariant foliations and consequently there exists a center foliation tangent to $E^c.$

In this paper we use the notation of $\phc^{r}_{c=1}$ for $C^r$-partially hyperbolic diffeomorhisms with uniformly compact one dimensional center leaves, that is the length of center leaves is uniformly bounded. 

We say a partially hyperbolic $f: M \rightarrow M$ is a topologically {\it Skew product} over (transitive)Anosov if $M$ is a circle bundle over $N$ such that $f: M \rightarrow M$ preserves a fibration $\pi: M \rightarrow N$ and projects to a transitive Anosov homeomorphism on $N$. Observe that the fibration need not to be trivial (See 4.1.1 in Bonatti-Wilkinson \cite{BW}). 

Pugh asked the following question: 
Is every partially hyperbolic diffeomorphism with compact center leaves finitely covered by a partially hyperbolic skew product?
(for a discussion on this question and some results see \cite{hammerlindl2017ergodic}).

We say $f$ is  \textit{virtually  skew product over Anosov} if $f$ is dynamically coherent and finitely covered by a skew product over Anosov. 

Any dynamically coherent partially hyperbolic diffeomorphism in three dimensional manifold with compact center leaves is a virtually skew product over Anosov (See \cite{bohnet2011partially}, \cite{martino}).
%if it admits an invariant center foliation $\mathcal{F}^c$ and $M$ is a fiber bundle over $M/\mathcal{F}^c$ and leaves of center foliation as fibers. 

%$(M, \mathcal{F}^c)$ is trivial if $\mathcal{F}^c$ is topologically equivalent to the product foliation of $M/\mathcal{F}^c \times S^1$. We say $(M, \mathcal{F}^c)$ is virtually trivial if after finite cover it is trivial. 

%In the virtually trivial case, 
%$f$ is called  a virtually trivial skew-product  if the quotient dynamics is topologically conjugate to a transitive Anosov diffeomorphism.

\begin{theorem}\label{teo2} 
 Let $f: M \rightarrow M$  be a $C^2-$partially hyperbolic  virtually  skew product over Anosov. If $M$ is the unique invariant minimal $u$-saturated set, then  one and only one of the following occurs:
\begin{itemize}
    \item[(1)]  $f$ admits a unique entropy maximizing measure $\mu$ and $\lambda_c(\mu) = 0$, or
    \item[(2)]   there are exactly two ergodic measures of maximal entropy $\mu^{+}, \; \mu^{-}$, with positive and negative center exponents, respectively.
\end{itemize}
\end{theorem}

Let us mention that in the above theorem we are not demanding minimality of unstable foliation, albeit the minimality would imply the hypothesis.

We recall that Guelman and Martinchich  \cite{guelman2022uniqueness} proved that for a skew product $f \in \phc^1_{c=1}(M)$ if the induced dynamics on $M / \mathcal{F}^c$ (quotient space of center foliation) is transitive and the center foliation is not a virtually trivial bundle then $f$ admits a unique minimal unstable lamination. We emphasize that the non trivial bundle hypothesis rules out many classical examples. In fact if the center foliation admits a global section then we are in the case of virtually trivial bundle. 

We also mention the work of Ures, Viana and Yang \cite{ures2021maximal} where they prove a similar dichotomy as in the above theorem for the partially hyperbolic diffeomorphisms on three dimensional nilmanifolds except for $\mathbb{T}^3$. A main property used in their proof is the uniqueness of minimal unstable lamination and accessibility in their setting, as proved by Hammerlindl-Potrie \cite{hammerlindlpointwise}. In our work, the main difficulty is in the trivial bundle case. We emphasize that we are not assuming accessibility of partially hyperbolic diffeomorphisms. So, the dichotomy given in \cite{hertz2012maximizing} is not available.

The proof of Theorem \ref{teo2} is based on a blend of arguments coming from invariance principle and properties of Margulis measures.

The Avila-Viana's invariance principle \cite{avila2010extremal} is used to understand the disintegration of measures of maximal entropy along center leaves. The first part of the dichotomy in the theorem is based on the invariance principle argument. However, as we have not accessibility property as in \cite{hertz2012maximizing}, we  exploit the minimality hypothesis. The second part of the dichotomy  mainly uses the construction of Margulis family of measures and their properties.

We emphasize that here we obtain exactly two ergodic hyperbolic maximizing measures. Without any minimality hypothesis, even in the presence of accessibility, one may cook up examples with more than two ergodic hyperbolic maximizing measures. By the way if besides accessibility, one assumes  topological transitivity and $f$ preserves the orientation of center leaves, an argument of Shi-Gan (Proposition 3.3 in \cite{yigan}) may yield to the same result as the above theorem \footnote{We thank Yi Shi and Shaobo Gan for communicating this to us.}. Clearly our minimality hypothesis implies topological transitivity.

Following Margulis, the autors in \cite{Ali} introduced the notion of \textit{Margulis measures system} $\{m_x\}_{x\in M}$ whose main feature is the existence of a constant  $D>0$ such that for every $x\in M$ we have  $f_{*}m_{x}=D^{-1}.m_{f(x)}$. The authors built a Margulis system $\{m^{u} _x\}_{x\in M^{u}}$ along unstable foliation for flow-type diffeomorphisms  with minimal stable and unstable foliations and they proved  that $f_{*}m_{x}^{u}=e^{-h_{top}(f)}.m_{f(x) }$. They used these measures to construct an invariant probability measure which was proved to be a  measure of maximal entropy and they obtain a striking dichotomy for diffeomorphisms of flow type (See Theorem 1.1 in \cite{Ali}).\\

In this work, we deal with the construction and characterization of maximal entropy measures for certain partially hyperbolic systems using the concept of Margulis measures.
We analyze the measures  of maximal entropy supported on invariant minimal $u$-saturated sets, i.e. compact sets that are $u$-saturated, $f$-invariant, and minimal with respect to these two properties.

A corollary of the proof of the above theorem is:
\begin{corollary}\label{corolario de teo}
If $M$ is the unique invariant minimal $u-$saturated set then $f$ admits a unique invariant minimal $s$-saturated set.
\end{corollary}

In principle $M$ can be the unique invariant minimal $u-$saturated set while the unique invariant minimal $s$-saturated subset is proper.

\begin{corollary}\label{outro corolario de teo}
If $M$ is the unique invariant minimal $u-$saturated set and there exists an ergodic measure of maximal entropy with vanishing center Lyapunov exponent then  $M$ is the unique invariant minimal $s-$saturated set too.
\end{corollary}

\section{Partial hyperbolicity and invariance principle}
In this section we review concepts of partial hyperbolicity,  disintegration of measures and invariance principle.

\subsection{General Definitions.} For a diffeomorphism $f: M \rightarrow M$ of a compact manifold to itself we recall the definition of norm and conorm with respect to a subspace of $V \subset$ $T_{x} M$ for some $x \in M:$
\begin{eqnarray*}
\|D f \mid V\|&:=&\max \{\|T f(v)\|: v \in V,\|v\|=1\} \text{ and }\\
\operatorname{conorm}(D f \mid V)&:=&\min \{\|T f(v)\|: v \in V,\|v\|=1\} .
\end{eqnarray*}
A splitting $E \oplus F$ is dominated  if it is nontrivial, $Df-$invariant, and there is some $N \geq 1$ such that, for all $x \in M$ :
$$
\left\|D f^{N} \mid E_{x}\right\|<\frac{1}{2} \operatorname{conorm}\left(D f^{N} \mid F_{x}\right)
$$

We also call $E^{cu}=E^{c}\oplus E^{u}$ and $E^{cs}=E^{c}\oplus E^{s}$.
It is a well-known fact that the strong bundles, $E^{s}$ and $E^{u}$, are uniquely integrable \cite{brin1974partially,invariantmanifolds}. That is, there are invariant strong foliations   $\mathcal{F}^{s}$ and $\mathcal{F}^{u}$ tangent, respectively, to the invariant bundles $E^{s}$ and $E^{u}$ (however, the integrability of $E^{c}$, $E^{cu}$ or $E^{cs}$ is a more delicate matter). We shall say that $f$ is dynamically coherent if there exist invariant foliations tangent to $E^{cu}$ and $E^{cs}$ (and then, to $E^{c}$).

\begin{definition}
    We say $f$ is accessible if given any two points $x, y \in M $ there exists a piecewise $C^1$ path tangent to $E^s \cup E^u$ from $x$ to $y$. More generally two points are in the same accessibility class if there is a path as above connecting the points. clearly, this is an equivalence class. By $AC(x)$ we mean the accessibility class containing $x.$
\end{definition}

In general, we will call $\mathcal{F}^{\sigma}$ any foliation tangent to $E^{\sigma}, \sigma=s, u, c, c s, c u$, whenever it exists and $\mathcal{F}^{\sigma}(x)$ the leaf of $\mathcal{F}^{\sigma}$ passing through $x$. A subset $K$ is $\sigma$-saturated if $\mathcal{F}^{\sigma}(x)\subset K$ for every $x \in K$. A closed subset $K\subset M$ is call $f$-invariant minimal $\sigma$-saturated set if:
\begin{itemize}
    \item[(1)] $K$ is $\sigma$-saturated set;
    \item[(2)] $K$ is an invariant set, i.e. $f(K)=K$;
    \item[(3)] If $K'\subset K$ is closed, $\sigma$-saturated and $f$-invariant set then $K'=\emptyset$ or $K'=K$.
\end{itemize}
We define
\begin{equation} \label{gammasigma} \Gamma^{\sigma}(f):=\{K\subset M: K \text{ is }f\text{-invariant minimal $\sigma$-saturated set}\}.
\end{equation}
Observe that we are taking invariant subsets which are $\sigma$-saturated and not just $\sigma$-saturated sets. Using Zorn's lemma we have that $\Gamma^{\sigma}(f) \neq \emptyset$.  

Given two points $x,y$ in the same leaf of $\mathcal{F}^{\sigma}$ ($\sigma \in \{s, u\}$) we say that $H^{\sigma}: \mathcal{F}^c(x) \rightarrow \mathcal{F}^c(y)$ is a global $\sigma-$holonomy if $H^{\sigma}(z) = \mathcal{F}^c(y) \cap \mathcal{F}^{\sigma}(z).$ 

\subsection{Quotient dynamics and holonomy}
Given a compact foliation $\mathcal{F}$ of a compact manifold $M$. We equip the quotient space $M_{\mathcal{F}}$ with the quotient topology. In the case of a center foliation, let $M_c:=M/\mathcal{F}^{c}$ be the space of the central leaves and $\pi:M\rightarrow M_c$ be the natural projection.  A Riemannian metric induces a Riemannian volume on the leaves of $\mathcal{F}$. The foliation $\mathcal{F}$ is called uniformly compact if all leaves are compact and
$$\sup_{x\in M} vol(\mathcal{F}(x))<+\infty.$$ 
When $\mathcal{F}^c$ is a uniformly compact foliation,  $M_c$ is a compact Hausdorff space.  $f$   induces a homeomorphism $f_c:M_c\rightarrow M_c$, $f_c(\pi(x))=\pi(f(x))$. $f_c$ is called \textit{quotient dynamic} of $f$. 
The main result of \cite{codimension1}  shows that if $f$ is a dynamically coherent partially hyperbolic diffeomorphism with compact central foliation and $\text{dim}(E^u)=1$ then  $\mathcal{F}^{c}$ is uniformly compact (See also \cite{carrasco2021}.)\\ 
%In fact by the work of Martinchich \cite{martinchich2023} the existence of a uniformly compact one dimensional invariant foliation tangent to the center bundle implies dynamical coherence.

%By 

In  $\dim(M)=3$, if $f$ is partially hyperbolic with compact center leaves, after considering finite cover we may assume that the invariant bundles are orientable and the dynamics preserves the orientations. In particular we have that $\mathcal{F}^c$ is a Seifert fibration without singular leaves over $\mathbb{T}^2$ and $f_c$ is topologically conjugate to an Anosov automorphism. See Theorem (3) in \cite{hertz2012maximizing}.

\begin{remark}
Let us just mention that by Bohnet \cite{bohnet2011partially} result even more general holds: If $f$ is a $C^1$-partially hyperbolic diffeomorphism with uniformly compact center foliation (any dimension),  $\dim(E^u)=1$ and $E^u$ is orientable, then $M_c$ is homeomorphic to torus and $f_c : M_c \rightarrow M_c$ is topologically conjugate to an Anosov automorphism. 
\end{remark}

Given two points $x, y$ in the same (un)stable leaf, $y \in \mathcal{F}^{\sigma}(x)
, \sigma \in \{s, u\}$ by transversality of $E^c$ and $E^{\sigma}$ there is a holonomy map defined in $U \subset \mathcal{F}^c(x)$, a neighbourhood of $x$ to a neighbourhood of $y$ homeomorphically. In general one may not have a global holonomy being a homeomorphism between $\mathcal{F}^c(x)$ and $\mathcal{F}^c(y)$. See for instance examples of center-stable foliations which are M\"{o}ebius band in Bonatti-Wilkinson \cite{BW} and there is no global homeomorphic holonomy. However, in the case of orientable bundles in three dimensions we have global holonomies which are homeomorphism. Indeed, center-stable or center-unstable leaves are cylinders and the holonomy map $H^{\sigma} :\mathcal{F}^c(x) \rightarrow \mathcal{F}^c(y) $ is defined  by $H(z) := \mathcal{F}^{\sigma}(z) \cap \mathcal{F}^c(y).$

\subsection{Measure Disintegrations and invariance principle} Let $M$ be a Polish space and $\mu$ be a finite Borel measure on $M$. Let $\mathcal{P}$ be a partition of $M$ into measurable sets. Let $\hat{\mu}$ be the induced measure on the $\sigma$-algebra generated by $\mathcal{P}$. A system of conditional measures of $\mu$ with respect to $\mathcal{P}$ is a family $\left\{\mu_{P}\right\}_{P \in \mathcal{P}}$ of probability measures on $M$ such that
 \begin{itemize}
     \item[(1)]  $\mu_{P}(P)=1$ for $\mu$-almost every $P \in \mathcal{P}$, and
     \item[(2)]  given any continuous function $\phi: M \rightarrow \mathbb{R}$, the function $P \mapsto \int \phi d \mu_{P}$ is integrable, and
$$
\int_{M} \psi d \mu=\int_{\mathcal{P}}\left(\int \phi d \mu_{P}\right) d \hat{\mu}(P)
$$
 \end{itemize}
Rokhlin \cite{rokhlin1949fundamental} proved that if $\mathcal{P}$ is a measurable partition, then the disintegration always exists and is essentially unique.  the The disintegration given by Rokhlin theorem varies measurably with the point of the quotient space. \\

A far reaching result to obtain more regularity of conditional measures in the dynamical setting has been obtained by Avila-Viana whose roots are in the work of Ledrappier \cite{Led}. 

We announce an  Invariance principle of Avila-Viana useful  for our purposes here.

\begin{theorem}[Teorema D, \cite{avila2010extremal}]\label{arturviana}
Let $f\in \phc_{c=1}^{1}(M)$. Suppose that given  $y\in \mathcal{F}^{\sigma}(x)$ a $\sigma$-holonomy  defined naturally between $\mathcal{F}^{c}(y)$ and $\mathcal{ F}^{c}(x)$ is a homeomorphism for $\sigma = s, u$. Let $\mu$ be a probability measure such that $\lambda_c(\mu)=0$ and $\mu^*=\pi_{*}\mu$ has a product structure. Then, $\mu$ admits a disintegration $\{\mu^{c}_{x^*}:x^*\in M_{c}\}$ which is $s$-invariant and $u$-invariant and whose conditional measures $\mu^{c}_{x^{*}}$ vary continuously with $x^{*}$ in the support of $\mu^*=\pi_{*}\mu$.
\end{theorem}

In our proof we use the above theorem for $\mu$ being measure of maximal entropy of $f$ (after lifts) which projects down to the measure of maximal entropy of $f_c$ which is topologically conjugate to Anosov automorphism and $\pi_*(\mu)$ has local product structure.

\section{Measures of maximal entropy, Margulis measures}

A very first result giving a dichotomy for the measures of maximal entropy of accessible partially hyperbolic diffeomorphisms with compact one dimensional center leaves is the following theorem.
\begin{theorem} \label{dichotomy} \cite{hertz2012maximizing}
Let $f : M \rightarrow M$ be a $C^{1+\alpha}$ partially hyperbolic diffeomorphism of a 3-dimensional closed manifold $M$. Assume that $f$ is dynamically coherent with compact one dimensional central leaves and has the accessibility property. Then $f$ has finitely many ergodic measures of maximal entropy. There are two possibilities:
\begin{enumerate}
\item  (rotation type) $f$ has a unique entropy maximizing measure $\mu$. The central Lyapunov exponent $\lambda_c(\mu)$ vanishes and $(f, \mu)$ is isomorphic to a Bernoulli shift,

\item (generic case) $f$ has more than one ergodic entropy maximizing measure, all of which with non-vanishing central Lyapunov exponent. The central Lyapunov exponent $\lambda_c(\mu)$ is nonzero and $(f, \mu)$ is a finite extension of a Bernoulli shift for any such measure $\mu.$ Some of these measures have positive central exponent and some have negative central exponent. \end{enumerate}
Moreover, the diffeomorphisms fulfilling the conditions of the second item form a $C^1-$open and $C^{\infty}-$dense subset of the dynamically coherent partially hyperbolic diffeomorphisms with compact one dimensional central leaves.
\end{theorem}

In the proof of the above theorem, the authors studied the disintegration of maximal entropy measures along center foliation. The first part of the above dichotomy is related to the so called invariance principle (Theorem \ref{arturviana}).
In the second case (generic case) one has an ergodic measure of maximal entropy with non-vanishing center Lyapunov exponent and the idea is to produce a new measure of maximal entropy with the opposite sign of center Lyapunov exponent. This correspondence is known as {\it twin-measure} construction. Let us assume that center foliation is one-dimensional with orientable compact leaves and suppose that $f$ preserves this orientation.

\begin{proposition} [\cite{hertz2012maximizing},  \cite{Ali}] \label{Twin measure} 
If $\eta$ is an ergodic measure of maximal entropy for $f \in \phc^2_{c=1}(M)$ with $\lambda_c(\eta) < 0$, then there is another $f$-invariant probability measure $\eta^*$ which is isomorphic to $\eta$ and with exponent $\lambda_c\left(\eta^*\right) \geq 0$. Moreover, there is a measurable set $Z \subset M$ with $\eta(Z)=1$ such that the restriction the next is map is an isomorphism:
$$
\beta: Z \rightarrow \beta(Z), x \mapsto \sup \mathcal{W}_s^c(x, f) \text { and } \eta^*=\beta_* \eta,
$$
where $\mathcal{W}_s^c(x, f):=\left\{y \in \mathcal{F}^c(x): \lim \sup _{n \rightarrow \infty} \frac{1}{n} \log d\left(f^n x, f^n y\right)<0\right\}$ and $\sup \mathcal{W}^c(x)$ is the extreme point of $\mathcal{W}^c(x)$ in the positive direction.
\end{proposition}

The above proposition uses in an essential way the one-dimensionality of center leaves. In the setting of Theorem \ref{dichotomy}, the authors proved that the twin measure $\eta^*$ is also hyperbolic, i.e $\lambda_c(\eta^*) > 0.$

\subsection{Entropy along unstable foliations}
A measurable partition  $\xi$ is called increasing and sub-ordinated to $\mathcal{F}$ if it satisfies:
\begin{itemize}
    \item[(a)] $\xi(x) \subseteq \mathcal{F}(x)$ for $\mu$-almost every $x$,
    \item[(b)] $f^{-1}(\xi) \geq \xi$ (increasing property),
    \item[(c)] $\xi(x)$ contains an open neighbourhood of $x$ in $\mathcal{F}(x)$ for $\mu$-almost every $x$.
\end{itemize}

The existence of a measurable increasing partition sub-ordinated to an invariant lamination in general is a delicate problem. For a uniformly expanding foliation invariant by a diffeomorphisms, we always have such a partition (See \cite{yang2016entropy}.) We say an invariant foliation $\mathcal{F}$ is uniformly expanding if there exists $\alpha>1$ such that $\forall x \in M,\left\|D f \mid T_{x}(\mathcal{F}(x))\right\|>\alpha$. Observe that if $\mathcal{F}$ is an expanding foliation, we have more useful properties for a partition $\xi$ which is increasing and sub-ordinated:
\begin{itemize}
    \item[(d)]  $\bigvee_{n=0}^{\infty} f^{-n} \xi$ is the partition into points;
    \item[(e)] the largest $\sigma$-algebra contained in $\bigcap_{n=0}^{\infty} f^{n}(\xi)$ is $\mathcal{B}_{\mathcal{F}}$ where $\mathcal{B}_{\mathcal{F}}$ is the $\sigma$-algebra of $\mathcal{F}$-saturated measurable subsets (union of entire leaves).
\end{itemize}

Let $\left\{\mu_{\xi(x)}\right\}_{x \in M}$ be the disintegration of $\mu$ w.r.t. $\xi$. The entropy with respect to $\mathcal{F}$ is defined \cite{LedrapierII} as:
$$h_{\mu}(f, \mathcal{F})=-\int \log \mu_{\xi(x)}\left(f^{-1} \xi(f x)\right) d \mu.$$
When  $f$ is partially hyperbolic system and $\mathcal{F}=\mathcal{F}^{u}$ we denote $h_{\mu}(f,\mathcal{F}^{u})$ by $h^{u}_{\mu}(f)$.
Denote the set of ergodic measures  by $\mathbb{P}_{erg}(f)$. The next result follows from Theorem $C^{\prime}$ in \cite{LedrapierII} and items (i)-(iii) after the statement of the theorem.

\begin{proposition}[Ledrappier-Young]\label{ledrapier}
Let $f \in\text{Diff}^{2}(M)$ be partially hyperbolic with strong unstable foliation $\mathcal{F}^{u} .$ For any $\mu \in \mathbb{P}_{\operatorname{erg}}(f)$,
$$
h_{\mu}^u\left( f\right) \leq h_{\mu}(f)
$$
If all center Lyapunov exponents of $\mu$ are non-positive, then the above inequality is an equality.
\end{proposition}

One can define the unstable topological entropy as:
$$h^u_{\topo}(f) = \sup\{ h^{u}_{\nu}(f): \nu \in \mathbb{P}_{erg}(f)\}.$$
This variational principle definition is equivalent to a topological definition which we will not repeat here. See (\cite{hu2017unstable}) for details. 
We say that a $f$-invariant measure  $\mu$ is measure of maximal $u$-entropy ($u$-mme) if only if 
$h^{u}_{\mu}(f)= h^u_{top}(f).$
Define
$$\mme^{u}(f)=\left\{\mu: \mu \text{ is an measure of maximal } u\text{- entropy}\right\}.$$ Similarly, define $MME^s(f)$ as $u$-maximal measures of $f^{-1}.$
We remember that a careful semi-continuity argument (See \cite{yang2016entropy}) shows that for any $C^1-$partially hyperbolic diffeomorphism (clearly no restriction on the center dimension), maximal unstable entropy measure always exist.

 Let  $f\in \phc^{2}_{c=1}(M)$, using one-dimensionality of center bundle, one can prove the following:

 \begin{proposition} \label{symmetricentropy}
     Let $f\in \phc^{2}_{c=1}(M)$, then $h^u_{top}(f) = h_{top}(f)$ and 
     $$\mme(f)= \mme^u(f) \cup \mme^s(f).$$
 \end{proposition}

\begin{proof}
    First let us show that $h^u_{top}(f) = h^u_{top}(f^{-1}) =  h_{top}(f)$. Indeed by definition, it is clear that $h^u_{top}(f) \leq h_{top}(f).$ Now take any ergodic maximal entropy measure $\mu$. If the central Lyapunov exponent of $\mu$ is non-negative, then using twin measure construction (Proposition \ref{Twin measure}) substitute it with another ergodic measure with the same entropy and non-negative center exponent. Now using above proposition \ref{ledrapier} we conclude that $h_{top}(f) = h_{\mu}(f)= h^u_{\mu}(f) \leq h^u_{top}(f).$
    So we end up proving that $h^u_{top}(f) = h_{top}(f).$ Asymmetric argument ends the proof of our claim.

    To conclude the proof of proposition, taker any $\mu \in MME(f).$ If $\lambda^c(\mu) \leq 0$ then using Proposition we get $h^u_{\mu} = h_{\mu}(f) = h_{top}(f) = h^u_{top}(f)$ which means $\mu \in MME^u(f).$ In the case where $\mu$ has positive center Lyapunov exponent we just change $f$ to $f^{-1}$ and conclude that $\mu \in MME^s(f).$

    \end{proof}

 \begin{remark}
  In the above proposition, we have shown that  $\mme^{u}(f)\subseteq \mme(f)$. It should be emphasized that this inclusion may be proper. See for instance section 4.3.1 in \cite{tahzibi2021unstable}  where the author gives an example of maximal entropy measure which is not $u-$maximal entropy. 
  We also emphasize that the inclusion $\mme^{u}(f)\subseteq \mme(f)$ uses the one-dimensionality of the center bundle. See in (\cite{ures2021maximal}) an example of a higher center dimension where this inclusion does not hold.
\end{remark}

%For example, in \cite{ures2021maximal}, in the case $f\in \phc^{2}_{c=1}(M^{3})$ is accessible we have $\mu\in \mme^{u }(f)$ if and only if $\lambda_c(\mu)\leq 0$. Therefore, for every measure of maximal entropy $\nu$ with $\lambda_c(\nu)>0$ we have $\nu\notin \mme^{u}(f)$.
 
\subsection{ Margulis system of measures}
 We will work with families of measures carried by the leaves invariant foliations.
 \begin{definition}
 Given a foliation $\mathcal{F}$ of some manifold $M$, a measurable system of measures on $\mathcal{F}$ is a family $\left\{m_{x}\right\}_{x \in M}$ such that:
 \begin{itemize}
     \item[(i)] for all $x \in M, m_{x}$ is a Radon measure on $\mathcal{F}(x)$;
     \item[(ii)]  for all $x, y \in M, m_{x}=m_{y}$ if $\mathcal{F}(x)=\mathcal{F}(y)$;
     \item[(iii)] $M$ is covered by foliation charts $B$ such that: $x \mapsto m_{x}\left(\phi \mid \mathcal{F}_{B}(x)\right)$ is measurable on $M$ for any $\phi \in C_{c}(B)$.
 \end{itemize}
 \end{definition}

The Radon property (i) means that each $m_{x}$ is a Borel measure and is finite on compact subsets of the leaf $\mathcal{F}(x)$ (here, and elsewhere, we consider the intrinsic topology on each leaf).\\

If $\left\{\mu_{x}\right\}_{x \in M}$ is the disintegration of some probability measure $\mu$ along a foliation $\mathcal{F}$ (understood as projective class) and if $\left\{m_{x}\right\}_{x \in M}$ is a measurable system of measures on $\mathcal{F}$, we will say that they coincide if  for $\mu$-a.e. $x \in M, \mu_{x}$ and $m_{x}$ are proportional.

\begin{definition}
Assume that $\mathcal{F}$ is a foliation which is invariant under some diffeomorphism $f: M \rightarrow M$, i.e., for all $x \in M, f(\mathcal{F}(x))=\mathcal{F}(f(x)) .$ A measurable system of measures $\{m_{x}\}_{x \in M}$ on $\mathcal{F}$ is dilated if there is some number $D>1$ such that for all $x \in M$ :
$$
f_{*} m_{x}=D^{-1} m_{f(x)}
$$
$D$ is called the dilation factor. We call the family $\left\{m_{x}\right\}_{x \in M}$ a Margulis system on $\mathcal{F}$ and the measures $m_{x}$ the Margulis $\mathcal{F}$-conditionals.
\end{definition}

We need also some invariance under holonomy for Margulis family of measures. 

\begin{definition}
Let $\mathcal{F}_{1}, \mathcal{F}_{2}$ be foliations which are invariant under some diffeomorphism $f \in \operatorname{Diff}^{1}(M) .$ Assume that $\left\{m_{x}\right\}_{x \in M}$ is a Margulis system of measures on $\mathcal{F}_{1}$ and that $\mathcal{F}_{2}$ is transverse to $\mathcal{F}_{1}$. The system $\left\{m_{x}\right\}_{x \in M}$ is invariant  along $\mathcal{F}_{2}$ if, for all $\mathcal{F}_{2}$-holonomies $h: U \rightarrow V$ with $U, V$ contained in $\mathcal{F}_{1}$-leaves :
\begin{eqnarray}
h_{*}\left(m_{x} \mid U\right)=m_{h(x)} \mid V\;\text{for any }\;x \in U.
\end{eqnarray}

\end{definition}

\section{Building Margulis system of measures}
In this section, we build Margulis measures on strong unstable and strong stable leaves  for  $C^{2}$ partially hyperbolic diffeomorphisms with compact one-dimensional center leaves. More precisely:
\begin{theorem}\label{existencia de u margulis}
 Let $M$ be a closed manifold and $f\in \phc_{c=1}^{2}(M)$. If the quotient 
map $f_{c}:M_{c}\longrightarrow M_{c}$ is  topologically transitive then there is a  measurable system (indeed continuous) of measures $\{m^{u}_{x}\}_{x\in M}$ such that:
 \begin{itemize}
      \item[(1)] each $m^{u}_x$ is a Radon measure without atoms and fully supported in $\mathcal{F}^{u}(x)$;
       \item[(2)] $f^{*}m^{u}_{f(x)}=e^{h_{top}(f)}.m_{x}^{u}$;
       \item[(3)] the system $\{m^{u}_{x}\}_{x\in M}$ is $cs$-invariant and locally uniformly finite.
  \end{itemize}
  Furthermore, if $\mu$ is an ergodic measure with $\lambda_c(\mu)\leq 0$, $\mu$ is of maximal entropy if and only if the disintegration $\{\mu^{u}_{ x}\}_{x\in M}$ of $\mu$ along unstable foliation is equivalent to $\{m^{u}_{x}\}_{x\in M}$.
\end{theorem}

\begin{remark}
We would like to mention that D. Bohnet \cite{bohnet2011partially}  had  also constructed Margulis family of measures in the context of partially hyperbolic diffeomorphisms with compact center leaves under trivial holonomy condition on center foliation and central transitivity of dynamics. Our method can be used for more general partially hyperbolic diffeomorphisms including perturbations of time-one map of Anosov flows. See also the result of Carrasco-Hertz \cite{CaHe} for Margulis measures.
\end{remark}
First, we build a system of Margulis measures supported on the center-unstable leaves. When the central foliation is uniformly compact, we deduce from this system a system of Margulis measures supported on the unstable leaves $\{m^{u}_x\}_{x\in M}$ which are invariant by center-stable holonomies. If the  quotient dynamics $f_c$ is topologically transitive we prove that $m^{u}_x$ has full support in $\mathcal{F}^{u} (x)$ for all $x\in M$. This will prove  Theorem \ref{existencia de u margulis}.

\subsection{The $cu$-conditionals}
Here we use Margulis' approach to construct a family of measures on invariant foliations transverse to a uniformly contracting (or expanding) foliation. 
\begin{proposition}\label{existencia mcu}
Let $M$ be a compact manifold and $f$ a $C^2$ partially hyperbolic and dynamically coherent. Then there exists a $cu$-system Margulis $\{m^{cu}_x\}_{x\in M}$, invariant by $ s$-holonomy and dilation constant $D_u>0$.
\end{proposition}

\begin{addendum}\label{adendum para uniformly compact}
In the theorem above if $\mathcal{F}^{c}$ is a uniformly compact foliation. Then
 \begin{itemize}
     \item[(a)] $m^{cu}_x(\mathcal{F}^{c}(x))=0$ for every $x\in M$,
     \item[(b)] $D_u>1$.
 \end{itemize}
\end{addendum}

\noindent\textbf{Functional on the set of $cu$ -functions:}
We denote by $\lambda^{\sigma}$ the intrinsic Riemannian volume on $\mathcal{F}^{\sigma}$ for $\sigma=u,s,cu,cs$. We denote by $d_{\sigma}$ the distance defined on each leaf of $\mathcal{F}^{\sigma}$ by the induced Riemannian structure and defining the intrinsic topology. The following definitions were given in \cite{Ali} for the construction of $cu$-Margulis systems using minimality of stable foliation.\\

The $\sigma$-balls are  $\mathcal{F}^{\sigma}(x,r):=\{y\in \mathcal{F}^{\sigma}(x):d_{\sigma}(y,x)<r\}$. For a subset $A$ of
such a leaf, we set  $B^{\sigma}(A,r):=\cup_{x\in A}\mathcal{F}^{\sigma}(x,r)$. A $\sigma$-subset is a $d_{\sigma}$-bounded subset of a $\sigma$-leaf. A $\sigma$-function   is a non negative function $\psi:M\rightarrow \mathbb{R}$ such that $\{x\in M:\psi(x)\neq 0\}$ is a $\sigma$-subset and the restriction of $\psi$ to $$\supp(\psi):=\overline{\{x\in M: \psi(x)\neq 0\}}$$
is continuous. We write $\psi>0$ if  $\psi\geq 0$ e $\{\psi>0\}$ has non-empty interior in the intrinsic topology.  We denote by $\mathcal{T}^{\sigma}$ the collection of all  $\sigma$-functions.\\
Given a $\sigma$-holonomy $h:A\rightarrow B$, its size is $\sup_{x\in A}d_{\sigma}(x,h(x))$,  and the two subsets $A$ e $B$ are called equivalent along  $\mathcal{F}^{\sigma}$ through $h$ if $h(A)=B$. We say that they are $\epsilon$-equivalent if the holonomy has size at most  $\epsilon$. Two functions $\psi_1$ e $\psi_2$ are $\epsilon$-equivalent along $\mathcal{F}^{\sigma}$  
if their supports are equivalent through a $\sigma$-holonomy $h$ with
size at most  $\epsilon$ and satisfying $\psi_2=\psi_1\circ h$. 
\vspace{0.7cm}

Following Margulis, we consider functionals $\lambda:\mathcal{T}^{cu}\rightarrow \mathbb{R}$. Note that $\lambda^{cu}$ is one such functional. The map $f$ acts on them by:
$$\forall \psi\in \mathcal{T}^{cu}\;\; f(\lambda)(\psi):=\lambda(\psi\circ f^{-1}),$$
A key class of such functionals are $\ell_{n}:=f^{n}(\lambda^{cu})$ for any $n\in \mathbb{N}$.That is, for any $\psi\in \mathcal{T}^{cu}$,
$$\ell_{n}(\psi):=\int \psi \circ f^{-n}d\lambda^{cu}.$$
We will use the following covering numbers. For $A$, a $cu$-subset, and $\epsilon>0$, we denote by  $r^{cu}(A,\epsilon)$ the smallest integer $k\geq 0$ such that there are $x_{1},\ldots,x_{k}\in A$ and $A\subset \cup_{i=1}^{k}\mathcal{F}^{cu}(x_{i},\epsilon)$.\\

The initial goal is to build a compact set of functionals on $\mathcal{T}^{cu}$. In the construction of $cu$- Margulis measures in \cite{Ali}, the authors used   the minimality of the stable foliation to define a normalizing $cu$-function which allows defining a compact set of functionals $\mathbb{L}$ on $\mathcal{T}^{cu}$. As we are not assuming  minimality of stable foliation, we need to strategically define normalizing $cu$-functions. For that we use bi-foliated boxes for $\mathcal{F}^{cu}$ and $\mathcal{F}^{s}$, more precisely: We fix  $B_{1},B_{2},\ldots,B_{l}$ bi-foliated box for  $\mathcal{F}^{s}$ and $\mathcal{F}^{cu}$  such that  there is functions  $g_{i}^{\sigma}:B_{i}\rightarrow B^{s}(0,\epsilon_1)\times B^{m-s}(0,\epsilon_1)$ and $\{(g_{i}^{\sigma},B_{i}):i=1,2,\ldots,l\}$ is a foliated atlas for  $\mathcal{F}^{\sigma}$, $\sigma=s,cu$. Note that   $$\bigcup_{i=1}^{l}B_i=M.$$

 To normalize, we fix  $\Phi_{1},\Phi_{2},\ldots,\Phi_{l}\in \mathcal{T}^{cu}$, as follows: for each $i=1,\ldots, l$ we take  $p_i\in B_i$,  letting $\Sigma_i=\mathcal{F}^{cu}_{B_i}(p_i)$ and $\Sigma=\Sigma_1\cup \ldots \cup \Sigma_l$. Note that  $\Sigma$ is complete  transversal to stable foliation. Now, we take non negative  functions $\Phi_{i}:\mathcal{F}^{cu}(p_i)\longrightarrow \mathbb{R}$ such that  $\phi_i|_{\Sigma_i}=1$ for $i=1,\ldots,l$. Considering the topology of pointwise convergence (i.e., working in $\mathbb{R}^{\mathcal{T}^{cu}}$ with the product topology), let $\mathbb{L}$ be the closure of the following set:
\begin{eqnarray*}\label{l1}
\mathbb{L}_{1}:=\Big\{\lambda=\sum_{i=1}^{n}c_{i}\ell_{t_i}:n\in \mathbb{N}^{*}, t_{1},\ldots , t_{n}\in \mathbb{N}, c_{1},\ldots, c_{n}>0 \textit{\; and \;}\sum_{i=1}^{l}\lambda(\Phi_{i})=1\Big\}
\end{eqnarray*}
 \begin{proposition}\label{Propjav}
 Let $f\in \diff^{2}(M)$ be partially hyperbolic with a splitting $TM=E^{s}\oplus E^{cu}$ and an invariant foliation $\mathcal{F}^{cu}$. Then there exist $\Lambda \in \mathbb{L}$ and $D_u>0$  such that $$f(\Lambda)=D_u.\Lambda$$ and, for some positive numbers $C_{_\Sigma}, r_{_\Sigma},$ for any $\psi\in \mathcal{T}^{cu}$:
 \begin{itemize}
     \item[(a)] $\Lambda(\psi)\leq C_{_\Sigma} r^{cu}(supp\psi,r_{_\Sigma})\|\psi\|_{\infty}$.
     \item[(b)] if $\phi\in \mathcal{T}^{cu}$ is $s$-equivalent  to $\psi$, then $\Lambda(\phi)=\Lambda_u(\psi)$. 
 \end{itemize}
 \end{proposition}
 To prove this, we will follow the steps of Theorem 4.2 in \cite{Ali} adapted to our setting. More precisely, we will show that $\mathbb{L}$ is a convex and compact set and then apply the Schauder-Tychonoff fixed point Theorem to a normalized action of $f$ on $\mathbb{L}$.\\
 
We will relate the iterations of different $cu$-functions by using the invariance
under $s$-holonomy and especially the following result (See proof of Theorem C in \cite{viana}).

\begin{theorem}\label{viana}
Let  $f$ be a $C^{2}$ diffeomorphism on a compact Riemannian manifold $M$.  Assume that there is a dominated splitting  $E^{s}\oplus E^{cu}$ with $E^{s}$ uniformly
contracting. Let $A_1,A_2$ submanifolds transverse to $\mathcal{F}^{s}$ and $(s,\delta)$-equivalent through $h:A_1\rightarrow A_2$. Let $\lambda_{1},\lambda_{2}$ the Riemannian volume $A_{1},A_{2}$, the measures $h_{*}\lambda_{1}$ and $\lambda_{2}$ are equivalent and there are   $\kappa(\delta)<\infty$  depending only on $f$ e $E^{s}$  such that:
$$|\dfrac{dh_{*}\lambda_{1}}{d\lambda_{2}}-1|\leq \kappa(\delta),\;\;\kappa(\delta)\rightarrow 0 \;\;\text{when}\; \delta \rightarrow 0.$$
\end{theorem}

Note that $\Sigma:=\Sigma_1\cup \ldots \cup \Sigma_l$ is a complete transversal section for the stable foliation. A result which is immediate consequence of Proposition 3.3 is the following:
\begin{lemma}\label{lemma1}
There are $C_{_\Sigma}<\infty$ and $r_{_\Sigma}>0$ such that 
 $$\forall x\in M,\; \forall n\geq 0,\; \lambda^{cu}(f^{n}B^{cu}(x,r_{_\Sigma}))\leq C_{_\Sigma} \sum_{i=1}^{l}\lambda^{cu}(f^{n}(\Sigma_{i})).$$
\end{lemma}

\begin{proof}
Let $x\in M$, there exists  $j\in \{1,2,\ldots,l\}$  such that $x$ is $s$-equivalent to some point in  $\Sigma_{j}$ for some $j$.
Fix $r_{1}>0$ small such that  $A_{j}:=\Sigma_{j}\backslash \mathcal{F}^{cu}(\partial \Sigma_{j},r_{1})$ is nonempty  and  $x$ is  $s$-equivalent to some point in $A_j$.   The continuity of the
foliation  $\mathcal{F}^{s}$ and its transversality to $\Sigma_j$  yield $r_{x}>0$ and $R_{x}<\infty$ such that,  for any  $y\in B(x,r_{x})$, $\mathcal{F}^{cu}(y,r_{x})$ is $(R_{x},s)$-equivalent   to a subset of  $\Sigma_j$. The compactness of $M$  yields $r_{_\Sigma}>0$, $R<\infty$ such that,  for any point $x\in M$, $\mathcal{F}^{cu}(x,r_{_\Sigma})$ é $(R,s)$-equivalent to a subset of some transversal section  $\Sigma\in \{\Sigma_i:i=1,2,\ldots,l\}$.\\
 Since $\mathcal{F}^{s}$ is contracted, there are numbers $C_{_\Sigma}<\infty$ and $\kappa<1$ such that the set $f^{n}(\mathcal{F}^{cu}(x,r))$ is $(\kappa^{n}R,s)$-equivalent to a subset of $f^{n}(\Sigma)$ for some  $\Sigma\in \{\Sigma_i:i=1,2,\ldots,l\}$. Theorem \ref{viana} proves the claim.
 \hspace{14cm}
\end{proof}

\begin{corollary}\label{alic45}
Let $C_{_\Sigma}<\infty$ and $r_{_\Sigma}>0$ as in the above lemma. Then  for every  $\psi\in \mathcal{T}^{cu}$ and  $n\geq 0$, we have:
$$\int\psi \circ f^{-n}d\lambda^{cu}\leq C_{_\Sigma} r^{cu}(\supp(\psi),r_{_\Sigma})\|\psi\|_{\infty}\sum_{i=1}^{l}\int \Phi_{i}\circ f^{-n}d\lambda^{cu}.$$
\end{corollary}
\begin{proof}
 The left hand side is bounded by $\|\psi\|_{\infty}.\lambda^{cu}(f^{n}(\supp \psi))$. Since $\supp(\psi)$ is a compact set,  then there are $x_1,\ldots,x_N\in \mathcal{F}^{cu}(\supp(\psi))$ with $N=r^{cu}(\supp(\psi),r_{_\Sigma})$ such that $\supp(\psi)\subset \cup_{i=1}^{N}B^{cu}(x_{i},r_{_\Sigma})$. Therefore,

\begin{eqnarray*}
\int\psi \circ f^{-n}d\lambda^{cu}&\leq&\|\psi\|_{\infty}.\lambda^{cu}(f^{n}(supp \psi))\\
&\leq&\|\psi\|_{\infty}\sum_{i=1}^{N}\lambda^{cu}(f^{n}(B^{cu}(x_{i},r_{_\Sigma})))\\
&\leq&\|\psi\|_{\infty}\sum_{i=1}^{N}C_{_\Sigma} \sum_{j=1}^{l}\lambda^{cu}(f^{n}(\Sigma_{j}))\\
&\leq&\|\psi\|_{\infty}\sum_{i=1}^{N}C_{_\Sigma}\sum_{j=1}^{l}\lambda^{cu}(\Phi_{j}\circ f^{-n})\\
&\leq&\|\psi\|_{\infty}N. C_{_\Sigma}\sum_{j=1}^{l}\lambda^{cu}(\Phi_{j}\circ f^{-n}).
\end{eqnarray*}
This concludes the proof of corollary.
\end{proof}
The next lemma establishes approximately holonomy invariance:
\begin{lemma}\label{lemabasse}
Let $\psi_1, \psi_2\in \mathcal{T}^{cu}$ cu-functions  $(s,\delta)$-equivalents for some  $\delta<\infty$, then for all  $\lambda\in \mathbb{L}$ we have
\begin{itemize}
    \item[(a)] $|\lambda(\psi_1)-\lambda(\psi_2)|\leq \kappa(\delta)\lambda(\psi_2)$,  $\kappa(\delta)\rightarrow 0$ when $\delta\rightarrow 0$.

    \item[(b)] if  $\psi_1 , \psi_2>0$ then there exists $C_0(\delta)\geq 1$ such that $ \lambda(\psi_1)\leq C_0(\delta).\lambda(\psi_2)$. Moreover $C_0(\delta)\rightarrow 1$ when $\delta \rightarrow 0$.
\end{itemize}
\end{lemma}
\begin{proof}
The proof of item (a) is entirely analogous to the proof of the first part of Lemma 4.6 in \cite{Ali}. To prove item (2), it is sufficient to take $C_0(\delta)=1+\kappa(\delta)$.
\end{proof}

\begin{proposition}\label{Cpositivo}
There is  $C>0$ such that 
 $\lambda(\psi\circ f^{-1})\geq C \lambda(\psi)$ for all $ \lambda\in \mathbb{L}$ and $\psi\in \mathcal{T}^{cu}$.
\end{proposition}
\begin{proof}
 Taking  $C=\min_{x\in M}|Jac(f|\mathcal{F}^{cu}(x))|$, we have that $C>0$ and
\begin{eqnarray*}
\ell_{k}(\psi\circ f^{-1})&=&\int\psi\circ f^{-k-1}d\lambda^{cu}\\
&=&\int\psi\circ f^{-k}. |Jac(f|\mathcal{F}^{cu})(x)|d\lambda^{cu}\\
&\geq& C \int\psi\circ f^{-k}d\lambda^{cu}\\
&=& C \ell_{k}(\psi),
\end{eqnarray*}
so,  for all $\lambda=\sum_{i\in I}c_i\ell_{k_i}$ we have that
\begin{eqnarray*}
\lambda(\psi\circ f^{-1})&=&\sum_{i\in I}c_i\ell_{k_i}(\psi\circ f^{-1})\\
&\geq & C \sum_{i\in I}c_i\ell_{k_i}(\psi)\\
&=&C \lambda(\psi).
\end{eqnarray*}
This extends to all $\lambda\in \mathbb{L}=\overline{\mathbb{L}_1}$. This concludes the proof of the proposition \ref{Cpositivo}.
\end{proof}

 \noindent  \textit{Proof of Proposition \ref{Propjav}.}\\
\textbf{ Step 1:} There exist $C_{_\Sigma},r_{_\Sigma}>0$  such that  $\lambda(\psi)\leq C_{_\Sigma}r^{cu}(\supp(\psi),r_{_\Sigma})\|\psi\|_{\infty}$ for all $\lambda \in \mathbb{L}$. Indeed, take $ C_{_\Sigma},r_{_\Sigma}>0$ as in the Lemma \ref{lemma1}.
By  Corollary \ref{alic45},  for all $\psi\in \mathcal{T}^{cu}$ and  $ n \geq 0$ we have: 
$$\ell_{n} (\psi)\leq C_{_\Sigma}r^{cu}(\supp\psi,r_{_\Sigma})\|\psi\|_{\infty}\sum_{j=1}^{l}\ell_{n}(\Phi_{j}).$$ Therefore, for $\lambda\in \mathbb{L}_{1}$:
\begin{eqnarray}\label{equatio8}
\lambda(\psi)=\sum_{i\in I}c_{i}\ell_{t_{i}}(\psi)&\leq& \sum_{i\in I}c_{i}C_{_\Sigma}r^{cu}(\supp\psi,r_{_\Sigma})\|\psi\|_{\infty}\sum_{j=1}^{l}\ell_{t_i}(\Phi_{j})\nonumber\\
&\leq& C_{_\Sigma}r^{cu}(\supp\psi,r_{_\Sigma})\|\psi\|_{\infty}\sum_{j=1}^{l}\lambda(\Phi_{j})\nonumber\\
&=& C_{_\Sigma}r^{cu}(\supp\psi,r_{_\Sigma})\|\psi\|_{\infty}.
\end{eqnarray}
This proves the claim since $\sum_{j=1}^{l}\lambda(\Phi_{j})=1$ for all $\lambda\in \mathbb{L}_{1}$. It extends to the closure  $\mathbb{L}$ concluding the proof of Step 1.\\

Note that equation (\ref{equatio8}) implies that $\mathbb{L}_{1}$ is a subset of the compact set (with product topology)
$$\prod_{\psi\in \mathcal{T}^{cu}}[0, C_{_\Sigma}.r^{cu}(supp \psi, r_{_\Sigma})\|\psi\|].$$
Hence its closure $\mathbb{L}$ is a compact subset of the locally convex linear space $\mathbb{R}^{\mathcal{T}^{cu}}$.\\

\noindent\textbf{Step 2:} \textit{ There exists  $\Lambda\in \mathbb{L}$ such that $f(\Lambda) = D_u \cdot \Lambda$} and $D_u >0.$\\
We now build the functional $\Lambda$ as a fixed point of the map
$$\overline{f}:\mathbb{L}\rightarrow \mathbb{L}, \lambda\mapsto \dfrac{\lambda(\;.\;\circ f^{-1})}{\sum_{i=1}^{l}\lambda(\Phi_i\circ f^{-1})}.$$

    We claim that $\bar{f}$ is well-defined and continuous from $\mathbb{L}$ to $\mathbb{R}^{\mathcal{T}^{c u}}$. Indeed, the  map $ \lambda \mapsto \sum_{i=1}^{l}\lambda\left(\Phi_{i} \circ f^{-1}\right)$ from $\mathbb{L}$ to $\mathbb{R}$ is well-defined since $\Phi_{i} \circ f^{-1} \in \mathcal{T}^{cu}$, is  continuous, and is positive, more precisely,  the  Lemma \ref{Cpositivo}  yields $C>0$ such that $$\sum_{i=1}^{l}\lambda\left(\Phi_{i} \circ f^{-1}\right)\geq \sum_{i=1}^{l}C.\lambda\left(\Phi_{i} \right)=C>0.$$
 Moreover, $\lambda\left(\cdot \circ f^{-1}\right): \mathcal{T}^{c u} \rightarrow \mathbb{R}$ is well-defined and $\lambda \mapsto \lambda\left(\cdot \circ f^{-1}\right)$ is continuous from $\mathbb{L}$ to $\mathbb{R}^{\mathcal{T}^{c u}}$ which proves the claim. Finally, it is obvious that $\bar{f}\left(\mathbb{L}_{1}\right)=\mathbb{L}_{1}$, hence $\bar{f}: \mathbb{L} \rightarrow \mathbb{L}$ is a well-defined continuous map. Since $\mathbb{L}$ is a convex, compact subset of the locally convex linear space $\mathbb{R}^{\mathcal{T}^{c u}}$, the Schauder-Tychonoff Theorem applies and yields $\Lambda \in \mathbb{L}$ with $\bar{f}(\Lambda)=\Lambda$.
In other words, $f(\Lambda)=D_u \cdot \Lambda$ with
$$D_u=\sum_{i=1}^{l}\Lambda\left(\Phi_{i} \circ f^{-1}\right)>0$$

\noindent\textbf{Step 3:}  $\Lambda$ is $s$-holonomy invariant.\\
The proof this claim is analogous  to the proof of $s$-invariance in  Proposition 4.2 from \cite{Ali}.\\

\noindent\textit{Proof of Proposition \ref{existencia mcu}.}
 Proposition \ref{Propjav}  yields a functional $\Lambda$ on $\mathcal{T}^{c u}$, which contains $C_{c}\left(\mathcal{F}^{c u}(x)\right)$ for all $x \in M .$ We note that, for each $x \in M, \mathcal{F}^{c u}(x)$ is a locally compact metric space and $\Lambda |_{C_{c}\left(\mathcal{F}^{c u}(x)\right)}$ is linear and positive (because this holds for all $\lambda \in \mathbb{L}_{1}$ and extends by continuity to $\mathbb{L}$ ). Hence, Riesz's Representation Theorem $[42,2.14]$ gives a Radon measure $m_{x}^{cu}$ on $\mathcal{F}^{c u}(x)$, for each $x \in M$, characterized by:
$$
\forall \psi \in C_{c}\left(\mathcal{F}^{c u}(x)\right), m_{x}(\psi)=\Lambda(\psi)
$$
The $s$-invariance of each $m_{x}^{cu}$ follows from $s$-invariance property (b) in Proposition \ref{Propjav}.

We prove that each $m_{x}^{c u}$ is atomless using the holonomy invariance. Indeed, assume by contradiction that there is $y \in \mathcal{F}^{c u}(x)$ with $m_{x}^{c u}(\{y\})>0$. Since $\mathcal{F}^{s}$ has no compact leaves and  $\Sigma= \Sigma_1 \cup \Sigma_2 \ldots \cup \Sigma_l$ is a complete transversal section for $\mathcal{F}^{s}$  then  there is $i_0\in \{1,\ldots,l\}$ such that    $\Sigma_{i_0} \cap \mathcal{F}^{s}(y)$ is a infinite set. By $s$-invariance we have that  $m^{cu}_{x_{i_0}}(\{z\})=m^{cu}_{x}(\{y\})$ for all $z\in \Sigma_{i_0} \cap \mathcal{F}^{s}(y)$. So,
$$m^{cu}_{x_{i_0}}(\overline{\Sigma_{i_0}})\geq \sum_{z\in \Sigma_{i_0} \cap \mathcal{F}^{s}(y)}m^{cu}_{x_{i_0}}(\{z\})= \sum_{z\in \Sigma_{i_0} \cap \mathcal{F}^{s}(y)}m^{cu}_{x}(\{y\})=+\infty.$$
 This contradicts the finiteness of $m_{x_{i_0}}^{cu}$ on compact sets.\\

The proof of the continuity comes from the holonomy invariance of $m^{cu}$ (See \cite{Ali}). \hspace{1cm}$\square$\\

\begin{remark}\label{S cu}
Let $\mathcal{M}^{cu}:=\{x\in M: x\in \supp m^{cu}_x \}$. Note that $f(\mathcal{M}^{cu})=\mathcal{M}^{cu}$ and,  by $s$-invariance it is easy to  see   that $\mathcal{M}^{cu}$  is an $s$-saturated and closed set. In fact, let $x\in\mathcal{M}^{cu}$ and $y\in \mathcal{F}^{s}(x)$. By continuity of $s$-foliation  there exists a $s$-holonomy $h^{s}:\mathcal{F}^{cu}(y,\delta)\rightarrow h^{s}(\mathcal{F}^{cu}(y,\delta)) \subset \mathcal{F}^{cu}(x)$ with $h^{s}(y)=x$. Given a open set $U_y\subset \mathcal{F}^{cu}(y)$ with $y\in U_y$, then:
$$m^{cu}_y(U_y)\geq m^{cu}_y(U_y\cap \mathcal{F}^{cu}(y,\delta))= m^{cu}_x(h^{s}(U_y\cap \mathcal{F}^{cu}(y,\delta)))>0.$$
This implies that $y\in \supp(m^{u}_y)$ and, therefore $\mathcal{M}^{cu}$ is $s$-saturated set. Now we prove that $\mathcal{M}^{cu}$ is closed set. Given a sequence $(x_n)_{n\in \mathbb{N}}\subset \mathcal{M}^{cu}$ such that $x_n\rightarrow z\in M$. Given a open set $U\subset \mathcal{F}^{cu}(z)$, since $x_n\rightarrow z$ a  continuity of the foliation $\mathcal{F}^{s}$  and its transversality to $U$ there is $n_0\in \mathbb{N}$ such that $\mathcal{F}^{s}_{loc}(x_{n_0})\cap U\neq \emptyset$. Take $y\in \mathcal{F}^{s}_{loc}(x_{n_0})\cap U$, since $x_{n_0}\in \mathcal{M}^{cu}$ and $\mathcal{M}^{cu}$ is $s$-saturated, then $y\in \mathcal{M}^{cu}$, and  since $y\in U$ we have that $m^{cu}_z(U)=m^{cu}_y(U)>0$. This implies that $\mathcal{M}^{cu}$ is a closed set.  
\end{remark}

\noindent\textit{Proof  of Addendum \ref{adendum para uniformly compact}.} Suppose that there exists a point $y\in M$ such that $m^{cu}_y(\mathcal{F}^{c}(y))>0$. Since $\mathcal{F}^{s}$ has no compact leaves then   there is a point $z\in M$ such that $\mathcal{F}^{cu}(z,1/2)\cap \mathcal{F}^{s}(y)$ is an infinite set. Then there exists a sequence of different points $(y_n)_{n\in \mathbb{N}}$ such that $y_n\in \mathcal{F}^{cu}(z,1/2)\cap \mathcal{F}^{s}(y)$. Since $\mathcal{F}^{c}$ is a uniformly compact foliation, Proposition 3.5 and Lemma 3.6 in \cite{correa2011compact} imply that  $\mathcal{F}^{s}(y_m)\cap \mathcal{F}^{c}(y_m)=y_{m}$, then  if $n\neq m$ then $y_n\notin \mathcal{F}^{c}(y_m)$. By $s$-invariance, we have that
 $$m^{cu}_z(\bigcup_{w\in \mathcal{F}^{cu}(z,1/2) } \mathcal{F}^{c}(w))\geq \sum_{n=1}^{\infty} m^{cu}_z(\mathcal{F}^{c}(y_n))=\sum_{n=1}^{\infty} m^{cu}_y(\mathcal{F}^{c}(y))=\infty.$$
This contradicts the finiteness of $m^{cu}_z$ on compact sets.\\

Now we show that $D_u>1$. Take $x\in M^{cu}$ and $A^{cu}(x,\delta):=\cup_{y\in\mathcal{F}^{c}(x)}\mathcal{F}^{u}(y,\delta) $ for small $\delta>0$. Consider the sequence $(f^{-n}(x))_{n\in \mathbb{N}}$. By compactness of $M^{cu}$, there is a point $p\in M^{cu}$ and a subsequence $(n_k)_{k\in \mathbb{N} } $ such that $\lim_{k\rightarrow \infty} f^{-n_k}(x)=p$. Thus, given $\epsilon>0$ there is $k_\epsilon$ such that $f^{-n_{k_{\epsilon}}}(A^{cu}(x,\delta))$ is $s$-equivalent to a subset $h^ {s} (f^{-n_{k_\epsilon}}(A^{cu}(x,\delta)))\subset A^{cu}(p,\epsilon)$. By $s$-invariance and regularity of $m^{cu}_p$ we have that
  \begin{eqnarray*}
  0=m^{cu}_p(\mathcal{F}^{c}(p))&=&\inf_{\epsilon>0} m^{cu}_{p}\Big( A^{cu} (p,\epsilon))\Big)\\
  &\geq&\inf_{\epsilon>0} m^{cu}_{p}\Big( h^{s} (f^{-n_{k_\epsilon}}(A^{cu}(x,\delta))))\Big)\\
  &\geq &\inf_{k\in \mathbb{N}} m^{cu}_{p}\Big( h^{s}(f^{-n_k}(A^{cu}(x,\delta)))\Big)\\
  &=&\inf_{k\in \mathbb{N}} m^{cu}_{f^{-n_k}(x)}\Big(f^{-n_k}(A^{cu}(x, \delta))\Big)\\
   &=&\inf_{k\in \mathbb{N}} D^{-n_k}_{u}m^{cu}_{x}\Big(A^{cu}(x,\delta)\Big ).
  \end{eqnarray*}
  Suppose on the contrary that $D_u\leq 1$ then $D_{u}^{-n_k}\geq 1$. Consequently,
   \begin{eqnarray*}
  0&=&\inf_{k\in \mathbb{N}} D^{-n_k}_{u}m^{cu}_{x}\Big(A^{cu}(x,\delta)\Big )\\
  &\geq&\inf_{k\geq k_0} m^{cu}_{x}\Big(A^{cu}(x,\delta)\Big)\\
  &=& m^{cu}_{x}\Big(A^{cu}(x,\delta)\Big)>0.\\
  \end{eqnarray*}
  This contradicts the fact that the center leaves have zero measure. Therefore $D_u>1$. \hspace{2.5cm} $\square$ \

\subsection{Building the $u$-conditionals.}  We complete the proof of Theorem \ref{existencia de u margulis}.
We start with the previously built Margulis $cu$-system $\{m^{cu}_x\}_{x\in M}$ and define
the family of measures $\{m^{u}_x\}_{x\in M}$ by extending subsets of $u$-leaves to subsets of
$cu$-leaves along the center foliation. Finally, we use the topological transitive of $f$ to prove that $m^{u}_x$ is fully supported in the unstable leave $\mathcal{F}^{u}(x)$ for all $x\in M$. \\

\noindent\textit{Proof of Theorem \ref{existencia de u margulis}.}\\
\noindent \textbf{Step 1:} \textit{ Inducing a Margulis family on unstable leaves.}\\
Here  we follow essentially the Margulis idea to induce a family of measures on unstable leaves using the already constructed Margulis family $m^{cu}.$ 

Since $\mathcal{F}^{c}$ is a uniformly compact foliation, Lemma 3.6 in \cite{correa2011compact} implies that  $\mathcal{F}^{u}(x)\cap \mathcal{F}^{c}(x)=\{x\}$ for all $ x\in M$. Now, for each $x \in M$ and Borel subset $A\subset \mathcal{F}^{u}(x)$  we let
$$m^{u}_x(A):=m^{cu}_{x}(\widehat{A})\;\; \text{with}\; \widehat{A}:=\bigcup_{y\in A}\mathcal{F}^{c}(y).$$
Then $m^{u}_x$ defines a measure on $\mathcal{F}^{u}(x)$: Observe that when $A$ is measurable then $\widehat{A}$ is measurable. Moreover, it is not difficult to check the countable additivity property of $m_x^u$ using the same property for $m_x^{cu}$. Indeed $\mathcal{F}^{c}(x)\cap \mathcal{F}^{u}(x)=\{x\}$ for all $x\in M$ and consequently $\widehat{A} \cap \widehat{B} = \emptyset$ if $A \cap B = \emptyset.$

Observe that relative compactness and having a nonempty interior, hold for $\widehat{A}$ if it holds for $A$. This implies  that $m^{u}_x$ is Radon like $m^{cu}_x$. Let $u$-sets $A, B\subset M $  $cs$-equivalents, the  dynamical coherence  implies  that $\widehat{A}$ and $\widehat{B}$ are $s$-equivalent. Then the $s$-invariance of  $\{m^{cu}_x\}_{x\in M}$ implies the $cs$-invariance of  $\{m^{u}_x\}_{x\in M}$. Moreover, since center leaves have zero measure, we concluded that  $m^{u}_x$ is atomless for every $x\in M$. Finally, note that 
$$m^{u}_{f(x)}(f(A))=m^{cu}_{f(x)}(\widehat{f(A)})=m^{cu}_{f(x)}(f(\widehat{A}))=D_u m^{cu}_x(\widehat{A})=D_u m^{u}_x(A).$$\\

\noindent \textbf{Step 2:} \textit{ The system is uniformly locally finite.}\\

Given $\delta>0$. For every  $x\in M$ take a $A^{cu}(x,\delta):= \widehat{\mathcal{F}^{u}(x,\delta)}$. Since  $\mathcal{F}^{c}$ is an uniformly compact foliation  then there exist  $R>0$ such that $A^{cu}(x,2\delta)\subset \mathcal{F}^{cu}(x,R)$ for all $x\in M$. Thus there  is $N_{\delta}\in \mathbb{N}$ such that
$$r^{cu}(A^{cu}(x,2\delta),r_{_{\Sigma}})\leq N_{\delta}, \text{ for all } x\in M.$$
Now, take a  continuous function $\phi_x:\mathcal{F}^{cu}(x)$ such that $\|\phi_{x} \|_{\infty}=1$,  $\phi_x|_{A^{cu}(x,\delta)}=1$ and $\supp(\phi_x)\subset A^{cu}(x,2\delta)$. Since $\phi_x\in \mathcal{T}^{cu}$ we have that
\begin{eqnarray*}
m^{cu}_{x}(A^{cu}(x,\delta/2))<m^{cu}_{x}(\phi_x)&=&\Lambda_{cu}(\phi_x)\\
&\leq& C_{_\Sigma}.r^{cu}(\supp\phi_x,r_{_\Sigma})\|\phi_x\|_{\infty}\\
&\leq& C_{_\Sigma}.N_{\delta}.
\end{eqnarray*}
This is the uniform local finiteness property.\\

\noindent \textbf{Step 3:} \textit{$u$-Margulis measures are fully supported on unstable leaves.}\\

Let $\mathcal{S}^{u}:=\{x\in M: x\in \supp(m^{u}_x)\}$. As in Remark \ref{S cu}, its easy to prove that  $f(\mathcal{S}^{u})=\mathcal{S}^{u}$ and $\mathcal{S}^{u}$  is a $cs$-saturated and closed set. 

%By definition of $\{m^{u}_{x}\}_{x\in M}$ we have that if $x \in \supp m_x^{u}$ then  $y \in \supp(m^{u}_y)$ for all $y\in \mathcal{F}^{c}(x)$. 
Since $f_c$ is topologically transitive there exists a point $q\in M_c$ such that $$cl(\mathcal{O}^{-}(q,f_c))=M_c.$$ Taking $p\in \pi_c^{-1}(q)$ we have that $\bigcup_{n\in \mathbb{N}}\mathcal{F}^{c}(f^{-n}(p))$ is dense in $M$. We prove that $p\in \mathcal{S}^{u}$.  Given $\delta>0$ then there is  $\theta>0$ such that $\mathcal{F}^{u}(x,\theta)\subset \widehat{\mathcal{F}^{u}(p,\delta)} $ for all $x \in \mathcal{F}^{c}(p)$.  Since $\{m^{u}_x\}_{x\in M}$ is a non-trivial Margulis system, there is $x_0\in M$  such that $x_0\in \supp(m^{u}_{x_0})$. By transversality of  $\mathcal{F}^{u}$ and $\mathcal{F}^{cs}$, there is $\epsilon>0$ such that  $x_0$ is $cs$-equivalent to a point  $h^{cs}(x_0)\in \mathcal{F}^{u}(y,\theta)$,  for every $y\in B(x_0,\epsilon)$. Since $\bigcup_{n\in \mathbb{N}}\mathcal{F}^{c}(f^{-n}(p))$ is dense in $M$ then there exist $N\in \mathbb{N}$ and  $y\in \mathcal{F}^{c}(f^{-N}(p))$ such that $d(y,x_0)<\epsilon$. So, $x_0$ is  $cs$-equivalent to a point  $h^{cs}(x_0)\in \mathcal{F}^{u}(y,\theta)$. Since $x_0\in \mathcal{S}^{u}$ and $\mathcal{S}^{u}$ is $cs$-saturated  then  $h^{cs}(x_0)\in \mathcal{S}^{u}$ and $m^{u}_{y}(\mathcal{F}^{u}(y,\theta))=m^{u}_{h^{cs}(x_0)}(\mathcal{F}^{u}(y,\theta))>0$. Taking $x=f^{N}(y)$ we have that
\begin{eqnarray*}
m^{u}_p(\mathcal{F}^{u}(p,\delta)) =m^{cu}_p(\widehat{\mathcal{F}^{u}(p,\delta)})
&\geq& m^{cu}_x(\widehat{\mathcal{F}^{u}(x,\theta)})\\
&=& m^{u}_x(\mathcal{F}^{u}(x,\theta))\\
&=& D_u^{-N} m^{u}_{f^{N}(x)}(f^{N}(\mathcal{F}^{u}(x,\theta)))\\
&\geq & D_u^{-N} m^{u}_{f^{N}(x)}(\mathcal{F}^{u}(f^{N}(x),\theta))\\
&= & D_u^{-N} m^{u}_{y}(\mathcal{F}^{u}(y,\theta))\\
&>&0.
\end{eqnarray*}
Since $\delta$ is arbitrarily small, we conclude that $p\in \supp(m^{u}_p)$, i.e. $p\in \mathcal{S}^{u}$. Since $\mathcal{S}^{u}$ is $c$-saturated and invariant set  we have that $\bigcup_{n\in \mathbb{Z}}\mathcal{F}^{c}(f^{n}(p))\in \mathcal{S}^{u}$. As $\mathcal{S}^{u}$ is closed set we have that $M=cl\big(\bigcup_{n\in \mathbb{Z}}\mathcal{F}^{c}(f^{n}(p))\big)\subset\mathcal{S}^{u}$. This implies that $\mathcal{S}^{u}=M$, i.e. $m^{u}_x$ has fully support in $\mathcal{F}^{u}(x)$ for all $x\in M$.\\

\noindent \textbf{Step 4:} \textit{ Continuity of the $u$-Margulis system}\\
    As $\mathcal{F}^{u}$ and $\mathcal{F}^{cs}$ are transverse, for any $x_0\in M$, there is a bi-foliated box $B$ at $x_0$. Then every unstable plaques in $B$ are $cs$-equivalents and so, given $x,y\in B$ there exists $cs$-holonomy $h^{cs}_{x,y}:\mathcal{F}^{u}_{B}(x)\rightarrow \mathcal{F}^{u}_{B}(y)$. Define $h^{cs}_{0}:B\times \mathcal{F}^{u}_{B}(x_0)\rightarrow B$ by $h^{cs}_{0}(x,z)=h^{cs}_{x_0,x}(z)$. Let $\phi\in C(M)$. By $cs$-holonomy invariance $m^{u}_x(\phi|\mathcal{F}^{u}_{B}(x))=m^{u}_x(\phi\circ h^{cs}_0(x,.))$. Hence,
\begin{eqnarray*}
\left|m_{x}^{u}\left(\phi \mid \mathcal{F}_{B}^{u}(x)\right)-m_{x_{0}}^{ u}\left(\phi \mid \mathcal{F}_{B}^{u}\left(x_{0}\right)\right)\right| \leq \int_{\mathcal{F}_{B}^{u}\left(x_{0}\right)}\left|\phi \circ h_{0}(x, \cdot)-\phi\right| d m_{x_{0}}^{u}
\end{eqnarray*}
which converges to $0$ as $x$ goes to $x_0$. This is the continuity property.\\

\noindent \textbf{Step 5:} \textit{$D_u =e^{h_{top}(f)}$ and $\mu^{u}_x\equiv m^{u}_x$ for every measure of maximal $u$-entropy $\mu$.}\\
The system $\{m^{u}_x\}_{x\in M}$ is $cs$-invariant, $D_u>1$ and $m^{u}_x$ is fully support on $\mathcal{F}^{u}(x)$ for all $x\in M$, continuing as in Lemma 5.10  from \cite{tahzibi2021unstable} we obtain that $\log(D_u)=\exp(h^u_{top}(f)$. Now by Proposition \ref{symmetricentropy}, $h_{top}(f) = h^u_{top}(f)=\log D_u$.

\section{Support of maximal measures} \label{section:Margulis measures: Support of maximal measures}

In this section, we analyze the support of maximal entropy measures and their uniqueness on minimal sets using the existence of Margulis families. Let us emphasize that the vast majority of the result about support of maximal measures in this section is for general partially hyperbolic diffeomorphisms and not just with compact center leaves.

\begin{proposition} \label{saturado}
    Let $f \in \diff^2(M)$ be a partially hyperbolic diffeomorphism with $u-$Margulis family of measures fully supported on unstable leaves for any $x \in M$. Then the support of any measure of maximal entropy with non-positive center Lyapunov exponents is $u-$saturated. In particular, if there are both $u$ and $s$ non-trivial Margulis families, then the support of any measure of maximal entropy with vanishing center Lyapunov exponents is $su-$saturated, i.e it is a union of accessibility classes.
\end{proposition}

\begin{proof}

To prove the first part of the above proposition, observe that by Proposition \ref{ledrapier} we conclude that any $\mu$ satisfying the hypothesis is $u-$mme.  
The following proposition  is proved in  \cite{Ali}. 

\begin{proposition}[Proposition  5.5, \cite{Ali}]\label{ProposAli}
Let $f \in \operatorname{Diff}^{2}(M)$ be a diffeomorphism with a Margulis $u$-system $\left\{m_{x}^{u}\right\}_{x \in M}$ with dilation factor $D_{u}>1$ and $m_{x}^{u}$ fully supported on $\mathcal{F}^{u}(x)$ for each $x \in M $. If $\mu$ is an ergodic measure,  then $\mu$ is u-mme i.e,   $h^u_{\mu}= h^u_{top}(f)$, if and only if the disintegration of $\mu$ along $\xi$ is given by $m_{x}^{u}/m_{x}^{u}(\xi(x)), \mu$-a.e. for any generating  increasing  partition  subordinated  to $\mathcal{F}^{u}$.
\end{proposition}

Let $\xi$ be a generating subordinated to $\mathcal{F}^u$ partition and $m^u_{\xi(x)} := \frac{m_x^u|_{\xi(x)}}{m^u(\xi(x))}$ (the normalized restriction.) Let $\{\mu^u_{x} \}$ be the disintegration of $\mu$ along $\xi.$ By the above proposition we have that $\mu^u_x = m^u_{\xi(x)}$ for $\mu-$ almost every $x \in M.$ So,
$$
1= \mu(\supp(\mu)) = \int \mu^u_{\xi(x)}(\supp(\mu)) d\mu(x) = \int m^u_{\xi(x)}(\supp(\mu)) d\mu(x). 
$$
The above equation implies that $m^u_{\xi(x)}(\supp(\mu)) =1$ for $\mu-$almost every $x.$ As $m^u_x$ is totally supported on $\xi(x)$ we have that $\xi(x) \subset \supp(\mu)$ for $\mu-$almost every $x.$ Consequently $\mu(A_0) = 1$ where $A_0 := \{x \in \supp(\mu) : \xi(x) \subset \supp(\mu)\}$. So, taking $A_k := f^k(A_0)$ $\mu(A_k) =1$ for all $k \geq 1.$ We conclude that
%As $\mu^u_{(f^k\xi)(x)} = f^k_{*} \mu^u_{\xi(x)}$ 
%we conclude that $\mu(A_k) =1$ where $A_k := \{ x \in A_{k-1} : (f^k\xi)(x) \subset \supp(\mu)\}.$ As $A_0 $ we have 
$$\mu(A^u) = 1 \quad \text{where} \quad A^u := \bigcap_{k=1}^{\infty} A_k.$$

By definition  $A^u = \{x \in \supp(\mu) : \mathcal{F}^u(x) \subset \supp(\mu)\}.$ As $A^u \subset \supp(\mu)$ and $\mu(A^u) =1$ we have $\supp(\mu) = cl(A^u).$ Now take any $ x \in \supp(\mu)$ then there exists a sequence of points $x_k \in A^u$ converging to $x$. As $\mathcal{F}^u(x_k) \in \supp(\mu)$ by continuity of unstable foliation we conclude that $\mathcal{F}^u(x) \in \supp(\mu).$

To finish the proof of the proposition, just apply the above argument for $f$ and $f^{-1}.$

\end{proof}

 Recall from subsection \ref{gammasigma} that $\Gamma^u(f)$ is the set of all closed invariant $u-$saturated subsets of $M$. 
\begin{lemma}\label{L1}
Let $f:M\rightarrow M$ be a $C^{2}$ partially hyperbolic diffeomorphism with a Margulis system $\{m^{u}_x\}_{x\in M}$, $cs$-invariant and $m^{u}_x$ fully supported on $\mathcal{F}^{u}(x)$ for all $x\in M$. If $K\in \Gamma^{u}(f)$ then there is $\mu \in MME^u(f)$ with $\supp(\mu) = K.$ In particular, if $f \in \phc^{2}_{c=1}(M)$ then $\mu$ is a measure of maximal entropy.
\end{lemma}
\begin{proof}
Following \cite{tahzibi2021unstable} we take $x_0\in K$ and  an open disk  $\gamma\subset \mathcal{F}^{u}(x_0)$ such that $m^{u}_{x_0}(\gamma)=1$ and let   $m^{u}_\gamma(E):=m^{u}_{x_0} (E\cap \gamma)$ for all $E\subset M$. Let $\mu_n=\dfrac{1}{n}\sum_{i=0}^{n-1}f^{j}_{*}m^{u}_{\gamma}$. By Lemma 5.10 from \cite{tahzibi2021unstable}  we have that any accumulation point $\mu$ of $(\mu_n)_{n\in \mathbb{N}}$ has conditional measure along  $\mathcal{F}^{u}$ coinciding with $\{m^{u}_x\}_{x\in M}$ $\mu$-a.e. $x\in M$. Therefore $h_{\mu}(f,\mathcal{F}^{u})=\log D_u$. This implies that $\mu \in MME^u(f).$ The proof of Proposition  \ref{saturado} implies that  $\supp(\mu)$ is $u$-saturated set. Moreover, $\supp(\mu)\subset K$. Indeed, Given $y\in M\backslash K$ then there exists $\epsilon>0$ such that $B(y,\epsilon) \cap K= \emptyset$. As $f^{n}(\gamma)\subset K$ for every $n\in \mathbb{N}$, we have that  $f^{n}(\gamma)\cap B(y,\epsilon)=\emptyset$ for all $n\in \mathbb{N}$. This implies that $\mu_n(B(y,\epsilon))=0$ for all $n\in \mathbb{N}$, then $\mu(B(y,\epsilon))=0$. So, $y\notin \supp(\mu)$ for every $y\in M\backslash K$, i.e. $\supp(\mu)\subset K$. Finally, as $\supp(\mu)$ is closed, $u$-saturated  and $f$-invariant set, we concluded that $\supp(\mu)=K$.
Finally using Proposition \ref{symmetricentropy} we complete the proof of the proposition.
\end{proof}

The next proposition is a corollary of Hopf-type argument using  properties of the Margulis family. 

 \begin{proposition}\label{PROPneg} If $\mu$ is an ergodic measure of $u$-maximal entropy with  $\lambda_c(\mu)<0$ and $ K :=\supp(\mu)  \in \Gamma^u(f),$  then $K \cap \supp(\nu)=\emptyset$ for every ergodic $u$-MME $\nu\neq \mu$. In particular,  $\mu$ is the unique $u$-$\mme$ for $f|_{K}$.
\end{proposition}
\begin{proof}
First, note that   $K=cl\big(\bigcup_{n\in  \mathbb{Z}}\mathcal{F}^{u}(f^{n}(x))\big)$ for every $x\in K$ since $K\in \Gamma^{u}(f)$. Let $\nu$ be an ergodic $u$-MME. Suppose that $\supp(\mu) \cap \supp(\nu)\neq \emptyset$. Take $x\in \supp(\mu) \cap \supp(\nu)$. As $\supp(\nu)$ is $u$-saturated and invariant closed set, we obtain  $$K=cl\Big(\bigcup_{\in \mathbb{Z}}\mathcal{F}^{u}(f^{n}(x))\Big)\subset \supp(\nu).$$
As $\nu$ is an ergodic measure of maximal $u$-entropy  the  Proposition \ref{ProposAli} implies that their conditional measures along unstable foliation for $\mu$ and $\nu$ are both given by the $u$-Margulis systems $\left\{m_{x}^{u}\right\}_{x \in M}$. Let $B_{\mu}:=\{x \in M:$ $\left.\frac{1}{n} \sum_{k=0}^{n-1} \delta_{f^{k} x} \rightarrow \mu\right\}$ be the ergodic basin of $\mu$ (the convergence is in the weak star topology as $n \rightarrow+\infty)$. Taking $K=\supp(\mu)$, by ergodicity $\mu\left(K \backslash B_{\mu}\right)=0$ and so $m_{x}^{u}\left(K \backslash B_{\mu}\right)=0$ for $\mu-$ a.e $x$. Similarly, letting $B_{\nu}$ be the ergodic basin of $\nu$, for $\nu-$ a.e $y$ we have $m_{y}^{u}\left(M\backslash B_{\nu}\right)=0$. \\

As the center Lyapunov exponent of $\mu$  is negative, for $\mu$ -a.e. $x\in B_{\mu}$,  $ m_{x}^{u}(K\backslash\left.B_{\mu}\right)=0$ for $\mu$-a.e. $x \in B_{\mu} $ and there is   $J \subset \mathcal{F}^{u}(x,1/2) \cap B_{\mu}$ with $m_{x}^{u}(J)>0$ and large size of the Pesin local stable manifolds: $\inf_{z\in J}diam^{cs}(\mathcal{W}_{\text {loc }}^{s}(z))>\delta$,  for some $\delta>0$. As $J \subset \mathcal{F}^{u}(x,1/2)$, there exists  $\epsilon>0$ such that for all $y\in B(x,\epsilon)$ there is a $(cs,\delta)$-holonomy $h^{s}:J \rightarrow h^{s}(J) \subset \mathcal{F}^{u}(y)$.\\

As  $K\subset\supp(\nu)$ then  there exists  $y \in B_{\nu}\cap B(x,\epsilon)$ with $m_{y}^{u}\left(M\backslash B_{\nu}\right)=0$. As $y \in B(x,\epsilon)$,  there is a local $cs$-holonomy $h: J \rightarrow \mathcal{F}^{u}(y)$. This holonomy is absolutely continuous from $\left(J, m_{x}^{u}\right)$ to $\left(\mathcal{F}^{u}(y), m_{y}^{u}\right)$, hence:
$$m_{y}^{u}(h(J))>0.$$
By the choice of $J, h(z)$ belongs to the Pesin stable manifold of $z$. Since the ergodic basin is saturated by stable manifolds, $h(J) \subset B_{\mu}$ and therefore $m_{y}^{u}\left(B_{\mu}\right)>0 .$ As $m_{y}^{u}\left(M \backslash B_{\nu}\right)=0$, we conclude that $B_{\nu} \cap B_{\mu} \neq \emptyset$ and consequently $\mu=\nu$. This  complete the proof of the proposition.
\end{proof}

\begin{remark}
    Observe that in the $3$-dimensional partially hyperbolic Kan example, there is an ergodic m.m.e which has positive center Lyapunov exponent which support is the whole manifold $\mathbb{T}^3$. In this example, there are two other m.m.e's with negative center exponent supported on invariant $su-$torus. Clearly, the above proposition implies that the measure with positive center exponent can not be $u$-m.m.e.
\end{remark}

%\begin{corollary}\label{corolario sconjuntos} \marginpar{como corolario do teorema 1.1}
%If  $\Gamma^{u}(f)=\{M\}$ and there exists an ergodic measure of maximal entropy with vanishing center exponent then  $\Gamma^{s}(f)=\{M\}$.
%\end{corollary}
%\begin{proof}
%Let $\mu$ be   an ergodic measures of maximal entropy with vanishing center exponent.
 %Since  $\Gamma^{u}(f)=M$ we have that   every ergodic measure of maximal entropy with non-positive center exponent has fully support, in particular $\supp(\mu)=M$.
 %If $\Gamma^{s}(f)\neq \{M\}$ then there exists $K^{s}\in \Gamma^{s}(f)=\Gamma^{u}(f^{-1})$  such that $K^{s}\neq M$. Applying the Lemma \ref{L1} we obtain a ergodic measure of maximal entropy $\nu$ with $\supp(\nu)=K^{s}$. Since $K^{s}\neq M$ and every ergodic measure of maximal entropy with non-positive center exponent has fully support we have that $\lambda_c(\nu)>0$.  The Proposition \ref{PROPneg} applied to $f^{-1}$ implies that $\supp(\nu)\cap \supp(\mu)=\emptyset$. So,
%$$K^{s}=K^{s}\cap M=\supp(\nu)\cap \supp(\mu)=\emptyset.$$
%This contradiction completes the proof of the corollary.
%\end{proof}
 
\begin{corollary}\label{corolario unici}
If $\Gamma^{u}(f)=\{M\}$ and  $\mu$ is a measure of maximal  $u$-entropy constructed as in Lemma \ref{L1} with $\lambda_c(\mu)<0$ then   $\mu$ is the unique measure of maximal  $u$-entropy.
\end{corollary}
\begin{proof}
Let $\nu\in \mathbb{P}_{\erg}(f)$ be an $u$-m.m.e. Since $\supp(\nu)$ is an  $u$-saturated set  and  $\Gamma^{u}(f)=\{M\}$ then $\supp(\nu)=M$. So, $\supp(\mu)\cap \supp(\nu)=M$, The proposition  \ref{PROPneg} implies that $\mu=\nu$.
\end{proof}

\section{Proof of theorem \ref{teo2}}
 
 Since $M$ is the unique minimal $u$-saturated invariant set  we have that $f$ is topologically transitive. So, the  quotient dynamic  $f_c$ is topologically transitive. Then the Theorem \ref{existencia de u margulis} applies to $f$ and $f^{-1}$ and  let    $\{m^{u}_x\}_{x\in M}$, $\{m^{ s}_x\}_{x\in M}$ be the Margulis system  as in the  Theorem \ref{existencia de u margulis} for $f$ and $f^{-1}$.\\
From now on, let us assume that $M$, $E^{u}$, $E^{s}$, and $E^{c}$ are orientable, and that $f$ preserves the orientation of these bundles. We will show that these assumptions are not restrictive. Assume that Theorem \ref{teo2} holds under these conditions. Now, let $f: M \rightarrow M$ be a diffeomorphism satisfying the hypotheses of Theorem \ref{teo2}. By taking a finite covering $\pi: \overline{M}\rightarrow M$, we can obtain a diffeomorphism $\overline{f}: \overline{M} \rightarrow \overline{M}$ such that $\overline{M}$ and the corresponding invariant bundles are orientable, and $g=\overline{f}^2$ preserves their orientations.

%We claim that $\overline{M}$ is also the unique $g$-invariant $u$-saturated minimal set. To see this, let us first check that $\overline{M}$ is the unique $u$-saturated $\overline{f}$-invariant minimal set. Indeed, let $n>0$ such that $\pi: \overline{M}\rightarrow M$ is a covering map with $n$ sheets. Suppose that $\{\overline{M}\}\neq\Gamma^{u}(\overline{f})$; then, there exists $\widetilde{x}_1\in \overline{M}$ such that $cl(\mathcal{F}^{u}(\mathcal{O}_{\overline{f}}(\widetilde{x}_1 )))\in \Gamma^{u}(\overline{f})$ and $cl(\mathcal{F}^{u}(\mathcal{O}_{\overline{f}}(\widetilde{x}_1 )))\neq \overline{M}$. Let $x=\pi(\widetilde{x}_1)$ and $\pi^{-1}(x)=\{\widetilde{x}_1, \widetilde{x}_2,\ldots,\widetilde{x}_n\}$. Then $cl(\mathcal{F}^{u}(\mathcal{O}_{\overline{f}}(\widetilde{x}_i )))$ and $cl(\mathcal{F}^{u}(\mathcal{O}_{\overline{f}}(\widetilde{x}_j)))$ are minimal proper $u$-saturated subsets disjoint from $\overline{M}$, for all $i,j\in\{1,2,\ldots,n\}$ with $i\neq j$. Since $cl(\mathcal{F}^{u}(\mathcal{O}_{\overline{f}}(\widetilde{x}_i )))$ projects onto $u$-saturated compact sets, it projects onto $M$. Thus, we have:$$\overline{M}=\bigsqcup_{i=1}^{n} cl(\mathcal{F}^{u}(\mathcal{O}_{\overline{f}}(\widetilde{x}_i))),$$  which implies that $\overline{M}$ coincides with the disjoint union of compact sets, leading to a contradiction, because $\overline{M}$ is connected.

\begin{claim}\label{minimalidad_levantamento}
$\overline{M}$ is the unique $u$-saturated minimal invariant set  for $\overline{f}$, i.e., $\Gamma^{u}(\overline{f})=\{\overline{M}\}$    
\end{claim}
\begin{proof}
Let $\Lambda \in \Gamma^{u}(\overline{f})$, we prove that $\Lambda$ is an open subset of $\overline{M}$. Since $\Lambda$ is a $u$-saturated minimal invariant set, we have $cl(\mathcal{F}^{u}(\mathcal{O}(z)))=\Lambda$ for every $z\in \Lambda$, and $\pi(\Lambda)= M$. Let $\widetilde{x}\in \Lambda$ and $U\subset\overline{M}$ be an open set containing $\widetilde{x}$, such that $\pi|_{U}:U\rightarrow \pi(U)$ is a homeomorphism. Let's prove that $U\subset \Lambda$. In fact, let $x=\pi(\widetilde{x})$, as $\pi(\mathcal{F}^{u}(\mathcal{O}(\widetilde{x})))=\mathcal{F}^{u}(\mathcal{O}(x))$ then 
\[
\pi(\mathcal{F}^{u}(\mathcal{O}(\widetilde{x}))\cap U)=\mathcal{F}^{u}(\mathcal{O}(x))\cap \pi(U)
\]
Since $cl(\mathcal{F}^{u}(\mathcal{O}(x)))=M$ (because $\Gamma^{u}(f)=\{M\}$) we have that $\mathcal{F}^{u}(\mathcal{O}(x))\cap \pi(U)$ is dense in $\pi(U)$ and, consequently $\mathcal{F}^{u}(\mathcal{O}(\widetilde{x}))\cap U=(\pi|_{U})^{-1}(\mathcal{F}^{u}(\mathcal{O}(x))\cap \pi(U))$ is dense in $U$. So,
\begin{eqnarray*}
U &=& cl(\mathcal{F}^{u}(\mathcal{O}(\widetilde{x}))\cap U) \\
&\subset& cl(\mathcal{F}^{u}(\mathcal{O}(\widetilde{x}))) \\
&=& \Lambda.
\end{eqnarray*}

As $\widetilde{x}\in\Lambda$ is an arbitrary point, it follows that $\Lambda$ is an open set. As $\Lambda$ is both open and closed, and $\Lambda\subset \overline{M}$, we conclude that $\Lambda=\overline{M}$. Hence, $\overline{M}$ is indeed the unique invariant $u$-saturated set.
\end{proof}

\begin{claim}\label{minimalidad_g}
 Let    $g=\overline{f}^{2}$ then  $\overline{M}$ is the unique $u$-saturated minimal invariant set  for $\overline{g}$.
\end{claim}
\begin{proof}
Suppose that $g=\overline{f}^{2}$ has a $g$-invariant and $u$-saturated minimal set $K\neq \overline{M}$. Hence, $K=cl(\mathcal{F}^{u}(\mathcal{O}_{g}(z)))$ for all $z\in K$. Since $K\cup \overline{f}(K)$ is a compact $\overline{f}$-invariant $u$-saturated set, we have $K\cup \overline{f}(K)=\overline{M}$. Thus, $C=K\cap \overline{f}(K)\neq \emptyset$ is a compact $\overline{f}$-invariant $u$-saturated set, and therefore $C=\overline{M}$. This implies that $K=\overline{M}$, and thus $\overline{M}$ is the unique $g$-invariant $u$-saturated minimal set. 
\end{proof}

Analogously to subsection 3.2 from \cite{hertz2012maximizing}, by Claims \ref{minimalidad_levantamento} and \ref{minimalidad_g} one can verify that the dichotomy of Theorem \ref{teo2} for $g=\overline{f}^{2}$ implies the dichotomy for $f$.\\

Now, suppose that there is an ergodic  measure of maximal entropy $\mu$ with $\lambda_c(\mu)=0$. Since $M$ is the unique  $u$-saturated minimal invariant set, then $\supp(\mu)=M$. By Proposition \ref{symmetricentropy}, $\mme(f)=\mme^{u}(f)\cup \mme^{u}(f^{-1})$, therefore  Proposition \ref{PROPneg} implies that there is no hyperbolic ergodic  measure of maximal entropy. Therefore, hyperbolic and non-hyperbolic ergodic measures of maximal entropy do not coexist, that is, only one of the following occurs:
\begin{itemize}
     \item $f$ admits an ergodic measure of maximal entropy with zero center exponent.
     \item $f$ admits no ergodic measure of maximal entropy with vanishing center exponent.
\end{itemize}

We will divide the proof of this theorem into two parts. In the first part, assuming the existence of an ergodic measure of maximal entropy  $\mu$ as $\lambda_c(\mu)=0$, we prove that it is unique. In the second part, assuming the non-existence of non-hyperbolic measures of maximal entropy, we show that there are exactly two ergodic measures of maximal entropy $\mu^{+}, \; \mu^{-}$, with a positive and negative central exponent, respectively.\\

\noindent\textit{\textbf{Part 1: Uniqueness of the non-hyperbolic maximal entropy measure.}}\\

To prove uniqueness, we apply the Avila-Viana invariance principle, Theorem \ref{arturviana} to conclude that $\mu$ admits $s$ and $u-$invariant disintegration along the central foliation $\{\mu^{c}_{x^*}\}_{x^* \in M_c}$ and $x^{*} \rightarrow \mu_{x^{*}}$ varies continuously with $x^{* }$ in $\supp\left(\pi_{*}(\mu)\right)=M_c$.

The following proposition is a consequence of uniqueness of Rokhlin's  disintegration theorem and invariance principle.

\begin{proposition}\label{PropoJoas}
Let $f\in \phc^{2}_{c=1}(M)$ and $\mu$ be an ergodic measure of maximal entropy with zero central exponent. Suppose that $f_c$ is topologically transitive and $\pi_{*} \mu$ has local product structure. If $\{\mu^{c}_{x^*}:x^*\in M_{c}\}$ is the disintegration of $\mu$ along $\mathcal{ F}^{c}$ then
\begin{itemize}
     \item[(1)] $ (f_c)_{*} \mu_{x^{*}}^{c}=\mu_{f_{c}(x^{*})}^{c}, \text{ for all } x^{*}\in M_{c}.$
     \item[(2)] $ \supp(\mu^{c}_{x^{*}}) \neq \emptyset , \text{ for all } x^{*}\in M_{c}.$
     \item[(3)] If $\Gamma^{u}=\{M\}$ then $\mu^{c}_{x^{*}}$ has no atoms and $\supp(\mu^ {c}_{x^{*}})=\mathcal{F}^{c}(x)$ for all $x\in M$, $x^{*}=\pi(x)$ .
\end{itemize}
\end{proposition}

\begin{proof} 
\textit{Item (1).} By the $f$-invariance of $\mu$ and the uniqueness of the disintegration we have that $(f_c)_{*} \mu_{x^{*}}^{c}=\mu_{f_{c}(x^ {*})}^{c}$ for $\pi_{*}\mu$-almost every point. Since $x^{*} \rightarrow \mu_{x^{*}}^{c}$ is continuous at $\supp\left(\pi_{*} \mu\right)=M_c$, we get that $ ( f_c)_{*} \mu_{x^{*}}^{c}=\mu_{f_{c}(x^{*})}^{c}$ for all $x^{*}\in M_c$.\\

\noindent\textit{Item (2).} Take $y^{*}\in M_{c}$ such that $\supp(\mu^{c}_{y^*})\neq \emptyset$. By continuing of disintegration $\{\mu^{c}_{x^*}:x^*\in M_{c}\}$, there exists an open $U\subset M_c$ such that $\supp(\mu ^{c}_{x^*})\neq \emptyset$ for all $x^{*}\in U$. Since $f_c$ is topologically transitive, there exists a point $p^{*}\in U$ such that $\supp(\mu^{c}_{p^*})\neq \emptyset$ and $cl(\mathcal{O}(p))=M_c$. Since $(f_c)_{*} \mu_{p^{*}}^{c}=\mu_{f_{c}(p^{*})}^{c}$ and $\mu^{c }_{p*}(\mathcal{F}^{c}(\pi^{-1}(p^{*})))=1$ we have $\mathcal{O}(p^ {*})\subset\{x^{*}\in M_c: \mu^{c}_{x*}(\mathcal{F}^{c}(\pi^{-1}( x^{*})))=1\}$. By continuing the disintegration and density of $\mathcal{O}(p^{*})$ we must have that $\mu^{c}_{x^{*}}(\mathcal{F}^{c}( x^{*}))=1,$ for all $x^{*}\in M_c$. This concludes the proof of item (2).\\

  \noindent\textit{Item (3).} Let $x^{*}\in M_c$, by item (2), $\mu^{c}_{x^*}$ is non-trivial. Let $x\in M$ such that $\pi_c(x)=x^{*}$, take  $p\in \supp(\mu^{c}_{x^*})$. For any $y^{*}\in M_c, y\in \pi^{-1}_{c}(y^{*})$ and $\epsilon>0$, we  show that $\mu^{c}_{y^{*}}(\mathcal{F}^{c}(y,\epsilon))>0$. For this, consider the disk $D^{cs}(y,\epsilon)$, defined as follows $$D^{cs}(y,\epsilon):=\bigcup_{z\in \mathcal{F} ^{ c}(y,\epsilon)}\mathcal{F}^{s}_{loc}(z).$$

Let $\pi^{s}_{loc}:D^{cs}(y,\epsilon) \rightarrow \mathcal{F}^{ c}(y,\epsilon)$ be the projection along the local stable manifolds.
Since $D^{cs}(y,\epsilon)$ is a disk transversal to $\mathcal{F}^{u}$ and $\mathcal{F}^{u}(\mathcal{O}_{f }(p))$ is dense in $M$, we have
that $ D^{cs}( y,\epsilon) \cap \mathcal{F}^{u}(\mathcal{O}_{f}(p))$ is dense in $D^{cs}( y ,\epsilon)$. So $$\#\Big(\pi^{s}_{loc}\big(D^{cs}(y,\epsilon)\cap \mathcal{F}^{u}(\mathcal{O}_f (p))\big)\Big)=\infty.$$

Let $z\in \mathcal{F}^{u}(\mathcal{O}_{f}(p))\cap D^{cs}( y,\epsilon)$, then there is $n_z\in \mathbb{Z}$, such that $z\in \mathcal{F}^{u}(f^{n_z}p)\cap D^{cs}( y,\epsilon)\neq \emptyset$, and $ \pi_{loc}^{s}(z)\in \mathcal{F}^{c}(y,\epsilon)$.
\begin{figure}[H]

\centering

\tikzset{every picture/.style={line width=0.7pt}} %set default line width to 0.75pt        

\begin{tikzpicture}[x=0.6pt,y=0.6pt,yscale=-1,xscale=1]
%uncomment if require: \path (0,192); %set diagram left start at 0, and has height of 192

%Shape: Ellipse [id:dp3412316026851103] 
\draw  [color={rgb, 255:red, 74; green, 74; blue, 74 }  ,draw opacity=1 ] (304.32,115) .. controls (292.21,73.03) and (300.58,39) .. (323.02,39) .. controls (345.45,39) and (373.45,73.03) .. (385.55,115) .. controls (397.66,156.97) and (389.29,191) .. (366.86,191) .. controls (344.43,191) and (316.43,156.97) .. (304.32,115) -- cycle ;
%Shape: Ellipse [id:dp399506248684246] 
\draw  [color={rgb, 255:red, 74; green, 74; blue, 74 }  ,draw opacity=1 ] (70.32,100) .. controls (58.21,58.03) and (66.58,24) .. (89.02,24) .. controls (111.45,24) and (139.45,58.03) .. (151.55,100) .. controls (163.66,141.97) and (155.29,176) .. (132.86,176) .. controls (110.43,176) and (82.43,141.97) .. (70.32,100) -- cycle ;
%Curve Lines [id:da9250364838555218] 
\draw [color={rgb, 255:red, 208; green, 2; blue, 27 }  ,draw opacity=1 ]   (1.4,99.2) .. controls (66.4,84.2) and (86.4,125.2) .. (153.4,112.2) .. controls (220.4,99.2) and (222.51,145.2) .. (265.46,131.2) ;
%Curve Lines [id:da41014970706438647] 
\draw [color={rgb, 255:red, 9; green, 79; blue, 161 }  ,draw opacity=1 ][line width=1.5]    (300.4,95.2) .. controls (303.8,106.4) and (306.4,126.2) .. (313.4,139.2) ;
%Curve Lines [id:da29891324949404496] 
\draw [color={rgb, 255:red, 74; green, 144; blue, 226 }  ,draw opacity=1 ]   (261.46,123.2) .. controls (295.53,93.2) and (315.5,103.2) .. (323.14,86.2) ;
%Curve Lines [id:da7905419869765127] 
\draw [color={rgb, 255:red, 3; green, 43; blue, 88 }  ,draw opacity=1 ]   (258.68,118.2) .. controls (292.75,88.2) and (307.51,102.2) .. (315.15,85.2) .. controls (322.79,68.2) and (333.69,113.2) .. (345.88,119.2) .. controls (340.84,134.2) and (308.85,129.2) .. (289.4,152.2) .. controls (289.45,132.2) and (264.61,148.2) .. (258.68,118.2) -- cycle ;
%Curve Lines [id:da2482132583095742] 
\draw [color={rgb, 255:red, 74; green, 144; blue, 226 }  ,draw opacity=1 ]   (270.4,135.2) .. controls (281.4,112.2) and (342.4,107.2) .. (330.32,94.2) ;
%Curve Lines [id:da6389475271434473] 
\draw [color={rgb, 255:red, 74; green, 144; blue, 226 }  ,draw opacity=1 ]   (287.4,146.2) .. controls (321.47,116.2) and (332.75,135) .. (340.39,118) ;
%Curve Lines [id:da1496744458285968] 
\draw [color={rgb, 255:red, 208; green, 2; blue, 27 }  ,draw opacity=1 ]   (320.4,116.2) .. controls (379.4,97.2) and (394.4,114.2) .. (433.4,95.2) ;
%Curve Lines [id:da3064629616178294] 
\draw [color={rgb, 255:red, 74; green, 144; blue, 226 }  ,draw opacity=1 ]   (283.4,141.2) .. controls (289.4,126.2) and (308.4,127.2) .. (336.4,107.2) ;
%Curve Lines [id:da8479641005655376] 
\draw    (311.4,158.2) .. controls (313.35,135.77) and (298.19,149.47) .. (305.4,127.74) ;
\draw [shift={(306,126)}, rotate = 109.56] [color={rgb, 255:red, 0; green, 0; blue, 0 }  ][line width=0.75]    (6.56,-1.97) .. controls (4.17,-0.84) and (1.99,-0.18) .. (0,0) .. controls (1.99,0.18) and (4.17,0.84) .. (6.56,1.97)   ;
%Curve Lines [id:da6288207527260312] 
\draw    (338.85,82.11) .. controls (324.85,99.73) and (345.05,96.68) .. (327.05,110.82) ;
\draw [shift={(325.58,111.94)}, rotate = 323.05] [color={rgb, 255:red, 0; green, 0; blue, 0 }  ][line width=0.75]    (6.56,-1.97) .. controls (4.17,-0.84) and (1.99,-0.18) .. (0,0) .. controls (1.99,0.18) and (4.17,0.84) .. (6.56,1.97)   ;

% Text Node
\draw (65,95) node [anchor=north west][inner sep=0.75pt]  [font=\Huge,color={rgb, 255:red, 138; green, 6; blue, 22 }  ,opacity=1 ]  {$\cdot $};
% Text Node
\draw (74,85.4) node [anchor=north west][inner sep=0.75pt]  [font=\footnotesize]  {$f^{n_{z}}( p)$};
% Text Node
\draw (124,47) node [anchor=north west][inner sep=0.75pt]  [font=\Huge]  {$\cdot $};
% Text Node
\draw (319,104) node [anchor=north west][inner sep=0.75pt]  [font=\huge,rotate=-15.36]  {$\cdot $};
% Text Node
\draw (304,98) node [anchor=north west][inner sep=0.75pt]    {$y$};
% Text Node
\draw (130,35.4) node [anchor=north west][inner sep=0.75pt]  [font=\footnotesize]  {$f^{n_{z}}( x)$};
% Text Node
\draw (300,115) node [anchor=north west][inner sep=0.75pt]  [font=\Huge]  {$\cdot $};
% Text Node
\draw (228,83) node [anchor=north west][inner sep=0.75pt]  [font=\footnotesize]  {$D^{cs}( y,\epsilon )$};
% Text Node
\draw (301.43,103) node [anchor=north west][inner sep=0.75pt]  [font=\huge,rotate=-15.36]  {$\cdot $};
% Text Node
\draw (302,158) node [anchor=north west][inner sep=0.75pt]  [font=\small]  {$\pi^{s}_{loc}(z)$};
% Text Node
\draw (340,67.4) node [anchor=north west][inner sep=0.75pt]  [font=\small]  {$z$};
% Text Node
\draw (312,20.4) node [anchor=north west][inner sep=0.75pt]  [font=\small]  {$\mathcal{F}^{c}(y)$};
% Text Node
\draw (53,4.4) node [anchor=north west][inner sep=0.75pt]  [font=\small]  {$\mathcal{F}^{c}\left( f^{n_{z}}( x)\right)$};
% Text Node
\draw (439,81.4) node [anchor=north west][inner sep=0.75pt]  [font=\small,color={rgb, 255:red, 143; green, 7; blue, 24 }  ,opacity=1 ]  {$\mathcal{F}^{u}\left( f^{n_{z}}( p)\right)$};
\end{tikzpicture}
\caption{Support of $\mu^{c}_{x^*}$ }
\end{figure}

Since $p\in \supp(\mu_{x^{*}}^{c})$, Item (1) implies that $f^{n_z}(p)\in \supp(\mu_{f^ {n_z}_c(x^*)} ^{c})$. Since $\pi^{s}_{loc}(z)\in AC(f^{n_z}(p))$ (accessibility class of $f^{n_z}(p)$) and the disintegration $[x^{*}\rightarrow \mu^{c}_{x ^{*}}]$ is $s$, $u$-invariant, we have $\pi^{s}_{loc}(z) \in \supp(\mu^{c}_{y^{ *}})$. In particular, if $p$ is an atom then $\mu(\pi^{s}_{loc}(z))=\mu_{y^*}(\{p\})=\alpha>0$. As $\#\big(\pi^{s}_{loc}(D^{cs}(y,\epsilon)\cap \mathcal{F}^{u}(\mathcal{O}_f(p) \big)=\infty$, then
\begin{eqnarray*}
    1=\mu^{c}_{y^{*}}(\mathcal{F}^{c}(y))&\geq& \mu^{c}_{y^{*}}\big( \pi^{s}_{loc}(D^{cs}(y,\epsilon)\cap \mathcal{F}^{u}(\mathcal{O}_f(p)\big)\\
    &=& \#\big(\pi^{s}_{loc}(D^{cs}(y,\epsilon)\cap \mathcal{F}^{u}(\mathcal{O}_f(p))\big). \alpha=\infty.
\end{eqnarray*}
This contradiction implies that $\mu^{c}_{x^{*}}$ cannot be atomic. Finally, as $z\in \mathcal{F}^{c}(y,\epsilon)$ and $\pi^{s}_{loc}(z)\in \supp(\mu^{c}_ {y^{*}})$ we have to
\begin{eqnarray}\label{weight}
\mu^{c}_{y^{*}}(\mathcal{F}^{c}(y,\epsilon))>0.
\end{eqnarray}
  Since $\epsilon>0$ is arbitrary, (\ref{weight}) implies that $y\in \supp(\mu^{c}_{y^*})$ for all $y\in \pi^{ -1}_c(y^{*})$, that is, the measure $\mu^{c}_{y^{*}}$ is fully supported. This completes the proof of the proposition.
\end{proof}

%%%%%%%%%%%%%%%%%%%%%%%%%%%%%%
%%%%%%%%%%%%%%%%%%%%%%%%%%
%%%%%%%%%%%%
\begin{remark}
In \cite{hertz2012maximizing} and \cite{ures2021maximal}, the main fact to prove uniqueness is that every conditional measure $\mu^{c}_x$ of a measure maximal entropy $\mu$ with $\lambda_c(\mu)=0$ has full support in $\mathcal{F}^{c}(x)$ and has no atoms, and for that, the authors  used the accessibility hypothesis. However, in Proposition \eqref{PropoJoas} we obtained the same, using the hypothesis $\Gamma^{u}(f)=M$ instead of accessibility.
\end{remark}

\noindent\textit{Proof of uniqueness.} From Proposition \ref{PropoJoas} using similar method as in \cite{avila2010extremal}, we obtain an $S^{1}$-action on $M$ that commutes with $f$ , that is, $\rho_{\theta}:M\rightarrow M$ such that $\rho_{\theta}\circ f = f\circ \rho_{\theta}, \theta \in S^1$. To do this, we define $\rho_{\theta}$ as follows: $\rho_{\theta}(x) = y$, where $y$ is the point in $\mathcal{F}^{c}(x) $ such that the center stable leaf arc that joins $x$ with $y$ (using the positive orientation) has conditional measure $\mu^{c}_x([x, y]_c) = \theta$ (we are identifying $S^1$ with $[0, 1]$ mod  $1$) and we are measuring in the positive direction.\\

Analogously to \cite{hertz2012maximizing} we conclude that $f$ is conjugate to $\hat{f}$, where $\hat{f}$ is a rigid rotation extension of the Anosov homeomorphism $f_{c}$. As $M$ is the unique  $u$-saturated invariant minimal set for $f$, by topological conjugacy we also have that $M$ is the unique  $u$-saturated invariant minimal set for $\hat{f}$.\ \

To show that $\mu$ is the unique  measure of maximal entropy, by the topological conjugacy of $f$ and $\hat{f}$ it is enough to show that $\hat{\mu}$ is the unique  measure of maximal entropy for $\hat{f}$. Since the quotient dynamics $f_c$ in $M/\mathcal{F}^{c}$ is a transitive hyperbolic homeomorphism, $f_c$ has a unique measure of maximal entropy that is locally the product of measures on the stable and unstable manifolds. Assume that $\nu$ is a  measure of maximal entropy  for $\hat{f}$. Then $\nu$ projects onto the unique measure of maximal entropy of $f_c$. As $M$ is the unique $u$-saturated minimal invariant set for $\hat{f}$, we have that conditional measures along central leaves have no atoms, are fully supported on $\mathcal{F}^{ c}$ and are invariant under  rotations. This implies that the conditional measures are Lebesgue and $\hat{\mu}=\nu$.\\

\noindent\textit{\textbf{Part 2: Twin measure construction and number of maximal measures}}\\

By the argument at the beginning of the proof of Theorem \ref{teo2}, let us assume that there are no measures of maximal entropy with zero central Lyapunov exponent.

Let $\mu$ be an ergodic maximal entropy measure with $\lambda_c(\mu)\neq 0$. We can assume that $\lambda_c(\mu)<0$. The Proposition \ref{Twin measure} implies the existence of another probability measure $f$-invariant $\nu$ which is isomorphic to $\mu$ and $\lambda_{c}(\nu) \geq 0$. By Corollary \ref{corolario unici} we have that $\mu$ is the only measure of maximal entropy with a non-positive central exponent. Thus, $\lambda(\nu)>0$.

Now we claim that $\nu$ is the unique ergodic measure of maximal entropy with positive center Lyapunov exponent. By contradiction, suppose there is another ergodic measure of maximal entropy $\eta$ with a positive central exponent. Substituting $f$ by $f^{-1}$ in the Proposition \ref{Twin measure} we obtain $\mu_1$, $\mu_2$ ergodic measures of maximal entropy for $f$ with  negative center Lyapunov exponents and $Z_1,Z_2\subset M$ such that $\nu(Z_1)=1=\eta(Z_2)$, $\beta_1:(Z_1,\nu)\rightarrow (\beta_1(Z_1),\mu_{1})$ and $\beta_2:(Z_2,\eta)\rightarrow (\beta_2(Z_2),\mu_2)$ as in Proposition \ref{Twin measure}. Due to the uniqueness of maximal  entropy with a negative central exponent (Proposition \ref{PROPneg}) we have that $\mu_1=\mu=\mu_2$. So $\mu(Z)=1$ where $Z=\beta_1(Z_1)\cap \beta_2(Z_2)$. Therefore, for all $z\in Z$ we have $\beta_{1}^{-1}(z),\beta_{2}^{-1}(z)\in (z,\sup \mathcal{ W}^{c}_{s}(z,f)]_c$. This implies that 
$$\beta_{2}^{-1}(z)\in \mathcal{W}^{c}_{s }(\beta_{1}^{-1}(z),f^{-1}) \text{ for }  \mu\text{-a.e. } z\in Z.$$

As forward and backward generic points of ergodic measures have full measure, We may take $z$ such that both $\beta_{2}^{-1}(z)$ and $\beta_{1}^{-1}(z)$ are respectively generic points of $\eta$ and $\nu.$ This implies that $\nu=\eta$.

Therefore, $f$ has exactly two ergodic measures of maximal entropy  $\mu^{-}:=\mu,$ and $\mu ^{+}=\nu$, where $\mu^{-}$ has a negative central exponent and $\mu^{+}$ has a positive central exponent. This concludes the proof of the first part of Theorem \ref{teo2}.\\

\subsection{Proof of Corollary \ref{corolario de teo}}

By contradiction, suppose that $\sharp \Gamma^{s}(f)>1$. Take $K_1, K_2 \in \Gamma^{s}(f)$ such that $K_1 \neq K_2$. It is clear that $K_i \neq M$ for $i=1,2$. By applying Lemma \ref{L1}, we conclude the existence of two distinct ergodic measures of maximal $u$-entropy for $f^{-1}$, denoted as $\mu_1$ and $\mu_2$, where $\operatorname{supp}(\mu_1)=K_1$ and $\supp(\mu_2)=K_2$, consequently $\mu_1 \neq \mu_{2}$. Since $\mme^{u}(f^{-1})\subseteq \mme(f^{-1})=\mme(f)$ we have that $\mu_1$ and $\mu_2$ are different measures of maximal entropy for $f$. By Theorem \ref{teo2}, one of those measures has  negative central Lyapunov exponent. We can suppose that $\lambda_{c}(\mu_1)<0$. By Proposition \ref{saturado} we have that $\supp(\mu_1)$ is a $u$-saturated set. Since $\Gamma^{u}(f)=M$, we have that $\supp(\mu_1)=M$. This implies that $K_1=M$. This contradiction implies that $\sharp \Gamma^{s}(f)=1$.

\hspace{12cm} $\quad\square$\\

\subsection{Proof of Corollary \ref{outro corolario de teo}}

Let $\mu$ be an ergodic measure of maximal entropy with vanishing center exponent. By Theorem \ref{teo2} $f$ has a unique measure of maximal entropy with full support and this also holds for $f^{-1}$. So, we conclude that $f^{-1}$ has no proper minimal invariant $u-$saturated set. Indeed, if this were the case, by Lemma \ref{L1}, $f^{-1}$ would admit a measure of maximal entropy with proper support.

\bibliographystyle{abbrv}
\bibliography{references}

\end{document}